\mag=\magstep1 
\input epsf 
\documentclass[10pt,a4paper]{amsart} 
\usepackage{amssymb,latexsym} 

\def\Cal{\mathcal}

\def\ot{\leftarrow} 
\def\<<{\langle } 
\def\>>{\rangle }

\numberwithin{equation}{section} 
\newtheorem{theorem}{Theorem}[section] 
\newtheorem{proposition}[theorem]{Proposition} 
\newtheorem{corollary}[theorem]{Corollary} 
\newtheorem{definition}[theorem]{Definition} 
 
\newtheorem{remark}[theorem]{Remark} 
\newtheorem{lemma}[theorem]{Lemma} 
 
\newtheorem{example}[theorem]{Example}

\def\Ker{\operatorname{Ker}} 
\def\Coker{\operatorname{Coker}}

\def\<{\langle} 
\def\>{\rangle}

\begin{document} 
 
\title{DG-category and Simplicial bar complex} 
 
\author{Tomohide Terasoma}

\maketitle 
 
\makeatletter 
\renewcommand{\@evenhead}{\tiny \thepage \hfill 
DG category and bar complex
\hfill}

\renewcommand{\@oddhead}{\tiny \hfill Tomohide Terasoma
 \hfill \thepage} 
\tableofcontents

\section{Introduction and conventions}
\subsection{Introduction}
In the paper \cite{C}, Chen defined 
a bar complex of an associative differential graded algebra
which computes the real homotopy type of a $C^{\infty}$-manifold.
There he proved that the Hopf algebra of the dual of 
the nilpotent completion
of the group ring $\bold R[\pi_1(X,p)]\hat\ $
of the fundamental group $\pi_1(X,p)$
is canonically isomorphic to the 0-th cohomology of
the bar complex of the differential graded algebra $A^{\bullet}(X)$
of smooth differential forms on $X$.

There are several methods to construct mixed Tate motives over a field
$K$. 
First construction is due to Bloch and Kriz \cite{BK1}, base on the bar
complex of the differential graded algebra of Bloch cycle complex.
They defined the category mixed Tate motives as the category of
comodules over the cohomology of the bar complex. 
Another construction
is due to Hanamura \cite{H}. 
In the book of Kriz-May \cite{KM},
they used a homotopical approach to define the category
of mixed Tate motives.
Hanamura used some generalization of complex 
to construct the derived category of mixed Tate motives.
One can formulate this construction of complex
in the setting of DG category,
which is called DG complex in this paper.
DG complexes is called twisted complexes in 
a paper of Bondal-Kapranov \cite{BK2}.
In paper \cite{Ke1}, similar notion of perfect 
complexes are introduced.
A notion of DG categories are also useful to study cyclic homology. 
(See \cite{Ke1}, \cite{Ke2}).

In this paper,
we study two categories, one is the category of comodules over the
bar complex of a differential graded algebra
$A$ and the other is the category of DG complexes
of a DG category arising from the differential graded algebra.
Roughly speaking, we show that these two categories are
homotopy equivalent (Theorem \ref{first main theorem}).
We use this equivalence to construct a certain coalgebra
which classifies nilpotent variation of mixed Tate Hodge
structres on an algebraic varieties $X$ 
(Theorem \ref{theorem bar and NVMTHS}). This coalgebra is 
isomorphic to
the coordinate ring of the Tanakian category of mixed 
Tate Hodge structures when $X=Spec(\bold C)$.

Bar construction is also used to construct the motives
associated to rational fundamental groups 
of algebraic varieties
in \cite{DG}. They used another type of bar construction due to
Beilinson. In this paper, we adopt this bar construction,
called simplicial bar construction, for differential graded algebras.
Simplicial bar complexes depend on the choices of two augmentations
of the differential graded algebras. 
If these two augmentations happen
to be equal, then the simplicial bar complex is quasi-isomorphic to 
the classical reduced
bar complex defined by Chen.

Let me explain in the case of the DGA of smooth differential
forms $A^{\bullet}$ of a smooth manifold $X$.
To a point $p$ of $X$, we can associate an augmentation
$\epsilon_p:A^{\bullet}\to \bold R$.
In Chen's reduced bar complexes,
the choice of the augmentation,
reflects the choice of the base point 
$p$ of the rational fundamental group.
On the other hand, simplicial bar complex depends on
two augmentations $\epsilon_1$ and $\epsilon_2$.
The cohomology of the simplicial bar complex of $A^{\bullet}$ 
with respect to the two augmentations arising from two points $p_1$ and
$p_2$ of $X$ is identified with the nilpotent dual of the linear hull
of paths connecting the points $p_1$ and $p_2$.

By applying simplicial bar construction
with two augmentations of cycle DGA 
arising from two realizations,
we obtain the dual of the space generated by
functorial isomorphism between
two realization functors. A comparison theorem gives a
path connecting these two realization functors.
Using this formalism, we treat
the category of variation of mixed Tate Hodge structures 
over smooth algebraic varieties. In this paper,
we show that the category of comodules over
the bar complex of the differential graded algebra of 
the Deligne complex of an algebraic variety $X$
is equivalent to the category of nilpotent variations of 
mixed Tate Hodge structures on $X$.

\subsection{Conventions}
\label{convention}
Let $\bold k$ be a field.
Let $\Cal C$ be a $\bold k$-linear abelian category with a tensor structure.
The category of complexes in $\Cal C$ is denoted as $K\Cal C$.
For objects
$A=(A^{\bullet},\delta_A)$ and
$B=(B^{\bullet},\delta_B)$ in $K\Cal C$, we define tensor product $A\otimes B$
as an object in $K\Cal C$
by the rule $(A\otimes B)^p=\oplus_{i+j=p}A^i\otimes B^j$.
The differential $d_{A\otimes B}$ on $A^i\otimes B^j$ is defined 
by 
$d_{A\otimes B}=\delta_A \otimes 1_B+(-1)^{i}1_A \otimes \delta_B$.
An element 
$$
\underline{Hom}^p_{K\Cal C}(A,B)=\prod_{i}Hom_{\Cal C}(A^i, B^{i+p})
$$ 
is called a homogeneous
homomorphism of degree $p$ from $A$ to $B$ in $K\Cal C$.
Let $A,A',B,B'\in K\Cal C$ and $\varphi=(\varphi_i)_i\in 
\underline{Hom}^p_{K\Cal C}(A,B)$ and 
$\psi=(\psi_j)_j\in \underline{Hom}^q_{K\Cal C}(A',B')$,
we define $\varphi\otimes \psi\in \underline{Hom}^{p+q}_{K\Cal C}
(A\otimes A',B\otimes B')$ by setting
$(\varphi\otimes \psi)_{i+j}=(-1)^{qi}\varphi_i\otimes \psi_j$
on $A^i\otimes {A'}^j\to B^{i+p}\otimes
{B'}^{j+q}$. (To remember this formula, the rule 
$(\varphi\otimes \psi)(a\otimes a')=
(-1)^{\deg(a)\deg(\psi)}\varphi(a)\otimes \psi(a')$
is useful.) An object $M$ in $\Cal C$ is regarded as an object
in $K\Cal C$ by setting $M$ at degree zero part.
For a complex $A\in K\Cal C$, we define the complex $A[i]$
by $A[i]^j=A^{i+j}$, where the differential is
defined through this isomorphism.
The shift $\bold k[i]$ of the unit object $\bold k$ 
is defined in this manner.
The homogeneous morphism 
$\bold k[i]\to \bold k[j]$ of degree $i-j$, whose degree $-i$ part
$\bold k[i]^{-i}=\bold k\to \bold k[j]^{-j}=\bold k$
is defined by the identity map, is denoted as $t_{ji}$. 
The degree $-i$ element ``$1$'' of $\bold k[i]$ is denoted as $e^i$.
For an object $B\in K\Cal C$,
the tensor complex $B\otimes \bold k[i]$ is denoted as $Be^i$.
For objects $A,B\in K\Cal C$ and 
$\varphi\in \underline{Hom}_{K\Cal C}^p(A,B)$,
a homomorphism $\varphi\otimes t_{ji}
\in \underline{Hom}_{K\Cal C}(Ae^i,Be^j)$
is a degree $(p+i-j)$ homomorphism.
As a special case, for $\varphi\in \underline{Hom}_{K\Cal C}^0(A, B)$,
the map 
$\varphi\otimes t_{i-1,i}\in \underline{Hom}_{K\Cal C}^1(Ae^i,Be^{i-1})$
is a degree one element. It is denoted as $\varphi\otimes t$ for simplicity.
The differential of $M$ can be regarded as a degree one map from 
$M$ to itself. Therefore $d$ can be regarded as an element in
$\underline{Hom}_{K\Cal C}^1(M, M)$.

An object in $KK\Cal C$ can be considered as a double complex
in $\Cal C$. Let
$(\cdots A^{\bullet i}\overset{d}\to A^{\bullet i+1}\to \cdots)$
be an object in $KK\Cal C$.
Since $d$ is a homomorphism of complex, we have
$$
d\otimes t\in
\underline{Hom}^1_{K\Cal C}(A^{\bullet,j}e^{-j},A^{\bullet,j+1}e^{-j-1}).
$$
is a degree one element in $K\Cal C$.
Let $\delta\otimes 1$ be the differential of $A^{\bullet i}e^{-i}\in K\Cal C$. 
The summation of $\delta\otimes 1$ for $i$ is also denoted as 
$\delta\otimes 1$.
Then the degree one map $\delta\otimes 1+d\otimes t$ 
becomes a differential on the total graded object. 
Here $\delta\otimes 1$ is called the inner differential and
$d\otimes t$ is called the outer differential.
Note that for an ``element'' $a\in A^{ij}$,
we have $(d\otimes t)(a\otimes e^{-j})=(-1)^id(a)\otimes e^{-j-1}$,
which coincides with the standard sign convention of 
the associate simple complex of a double complex.

The resulting complex
is called the associate simple complex of $A^{\bullet\bullet}\in KK\Cal C$
and denoted as $s(A)=s(A^{\bullet,\bullet})$.
For objects $A=A^{\bullet,\bullet}, B^{\bullet,\bullet}\in KK\Cal C$, 
the tensor
product $A\otimes B$ is defined as an object in $KK\Cal C$.
Then we have a natural isomorphism in $K\Cal C$ 
\begin{equation}
\label{tensor of double complex to simple complex}
\nu:s(A)\otimes s(B)\simeq s(A\otimes B)
\end{equation}
defined by $\nu(ae^{-j}\otimes be^{-j'})=(-1)^{ji'}(a\otimes b)e^{-j-j'}$ for
$ae^{-j}\in A^{ij}e^{-j},be^{-j'}\in B^{i'j'}e^{-j'}$. 
This isomorphism is compatible with
the natural associativity isomorphism.


\section{DG category}
\subsection{Definition of DG category and examples}
Let $\bold k$ be a field.
A DG category $\Cal C$ over $\bold k$ consists of the following
data
\begin{enumerate}
\item A class of objects $ob(\Cal C)$.
\item
A complex 
$
\underline{Hom}_{\Cal C}^{\bullet}(A,B)=
(\underline{Hom}_{\Cal C}^{\bullet}(A,B),\partial)
$
of $\bold k$ vector spaces
for every objects $A$ and $B$ in $ob(\Cal C)$.
\vskip 0.1in
\noindent
We sometimes impose the following shift structure on $\Cal C$.

\item
Bijective correspondence $T:\Cal C\mapsto \Cal C$ for objects in $\Cal C$.
An object
$T^k(A)$ in $\Cal C$ is denoted as $Ae^k$ for $k\in \bold Z$.
\end{enumerate}
with the following axioms.
\begin{enumerate}
\item
For three objects $A,B$ and $C$ in $\Cal C$, the composite
$$
\underline{Hom}_{\Cal C}^{\bullet}(B,C)
\otimes \underline{Hom}_{\Cal C}^{\bullet}(A,B)
\to
\underline{Hom}_{\Cal C}^{\bullet}(A,C)
$$
is defined as a homomorphism of complexes over $\bold k$.
\item
The above composite homomorphism is associative.
That is, the following diagram of complexes commutes.
$$
\begin{matrix}
\underline{Hom}_{\Cal C}^{\bullet}(C,D)\otimes
\underline{Hom}_{\Cal C}^{\bullet}(B,C)\otimes 
\underline{Hom}_{\Cal C}^{\bullet}(A,B)
& \to &
\underline{Hom}_{\Cal C}^{\bullet}(C,D)\otimes
\underline{Hom}_{\Cal C}^{\bullet}(A,C) \\
\downarrow & & \downarrow \\
\underline{Hom}_{\Cal C}^{\bullet}(B,D)\otimes 
\underline{Hom}_{\Cal C}^{\bullet}(A,B)
& \to &
\underline{Hom}_{\Cal C}^{\bullet}(A,D) 
\end{matrix}
$$
\item
There is a degree zero closed element $id_A$ in 
$\underline{Hom}^0(A,A)$ for each $A$,
which is a left and right identity under the above 
composite homomorphism.
\vskip 0.1in
\noindent
If we assume the shift structure $T$, the following sign convention
should be satisfied.

\item
There is a natural isomorphism of complexes
$$
\underline{Hom}_{\Cal C}^{\bullet}(A,B)[-i+j]\simeq
\underline{Hom}_{\Cal C}^{\bullet}(Ae^i,Be^j):
\varphi\mapsto \varphi\otimes t_{ji}
$$
satisfying the rule
$(\varphi\otimes t_{ji})\circ(\psi\otimes t_{ik})
=(-1)^{(i-j)\deg(\psi)}(\varphi\circ\psi)\otimes t_{jk}$.
(It is compatible with the formal 
commutation rule for $\psi$ and $t_{ji}$.)
\end{enumerate}
\begin{definition}
\begin{enumerate}
\item
Let $\Cal C$ be a DG category and $a,b$ objects in $\Cal C$.
A closed morphism $\varphi:a\to b$ of degree $0$
(i.e. $\partial\varphi=0$) is called an isomorphism if
there is a closed morphism $\psi$ of degree zero such that
$\psi\circ \varpi=1_a,\varphi\circ \psi=1_b$.
\item
Let $\Cal C_1,\Cal C_2$ be DG categories. A DG functor $F$ is 
a collection $\{F(a)\}_a$ of objects in $\Cal C_2$ indexed by objects
in $\Cal C_1$ and a collection $\{F_{a,b}\}$ of homomorphisms
of complexes 
$F_{a,b}:\underline{Hom}_{\Cal C_1}^{\bullet}(a,b)
\to\underline{Hom}_{\Cal C_2}^{\bullet}(F(a),F(b))$
indexed by $a,b\in \Cal C_1$, which preserves the composites,
identities and degree shift operator $A\mapsto A[1]$.
We define sub DG categories
and full sub DG categories similarly as in usual categories.
We also define essentially 
surjective functors as in the usual category.
A functor is equivalent
if and only if it is essentially surjective and fully faithful.
\item
A DG functor $F:\Cal C_1 \to \Cal C_2$ is said to be homotopy
equivalent if and only if it is essentially surjective, and
the induced map
$$
H^i(F_{a,b}):H^i(\underline{Hom}_{\Cal C_1}^{\bullet}(a,b))
\to H^i(\underline{Hom}_{\Cal C_2}^{\bullet}(F(a),F(b)))
$$
is an isomorphism for all $i\in \bold Z$ and $a,b\in \Cal C_1$.

\end{enumerate}
\end{definition}
\begin{example}
Let $Vec_{\bold k}$ be a category of $\bold k$-vector spaces. 
The category of complexes
of $\bold k$-vector spaces is denoted as $KVec_{\bold k}$.
Then $KVec_{\bold k}$ becomes a DG category by setting
$$
\underline{Hom}_{KVec_{\bold k}}^{p}(A,B)=\prod_{i}Hom(A^i,B^{i+p})
$$
for complexes $A=A^{\bullet}$ and $B=B^{\bullet}$.
The differential $\partial \varphi$ of an element 
$\varphi=(\varphi_{i})_{i}\in 
\underline{Hom}_{KVec_{\bold k}}^p(A,B)$ 
is defined by the formula
\begin{equation}
\label{differential formula}
(\partial(\varphi))_{i}=
d_B\circ\varphi_{i} - (-1)^{p}\varphi_{i+1}\circ d_A.
\end{equation}
Therefore $\varphi\in \underline{Hom}_{KVec_{\bold k}}^0(A,B)$ is a homomorphism
of complexes if and only if $\partial(\varphi)=0$.
Two homomorphisms of complexes $\varphi$ and $\psi$ are homotopic
to each other
by the homotopy $\theta$ if and only if
$\varphi-\psi =\partial(\theta)$ with
$\theta \in \underline{Hom}_{KVec_{\bold k}}^{-1}(A,B)$.
\end{example}

\begin{definition}[DG category associated with a DGA]
Let $A=A^{\bullet}$ be a unitary associative 
differential graded algebra (denoted as DGA for short)
over a filed $\bold k$ with the multiplication 
$\mu:A^\bullet \otimes A^\bullet \to A^{\bullet}$.
We define a DG category $\Cal C_A$ associated to $A$ as
follows. 
\begin{enumerate}
\item
An object of $\Cal C_A$ is a complex $V=V^{\bullet}$
of vector spaces over $\bold k$.
\item
For two objects $V=V^{\bullet}$ and 
$W=W^{\bullet}$, the set of morphisms $\underline{Hom}^p_{\Cal C_A}(V,W)$
is defined as
$$
\underline{Hom}^{\bullet}_{\Cal C_A}(V,W)=
\underline{Hom}_{KVec_{\bold k}}(V^{\bullet},
A^{\bullet}\otimes W^{\bullet}).
$$
Then $\underline{Hom}_{\Cal C_A}^{\bullet}(V,W)$ becomes a complex
by the formula (\ref{differential formula}) and the structure of
tensor complex $A^{\bullet}\otimes W^{\bullet}$ 
defined in (\ref{convention}).

\item
For three objects $U=U^{\bullet},V=V^{\bullet}$ and 
$W=W^{\bullet}$, we define the composite
\begin{align*}
\mu: & \underline{Hom}_{\Cal C_A}^{\bullet}(V^{\bullet},W^{\bullet})\otimes
\underline{Hom}_{\Cal C_A}^{\bullet}(U^{\bullet},V^{\bullet}) 
\to
\underline{Hom}_{\Cal C_A}^{\bullet}(U^{\bullet},W^{\bullet}) 
\end{align*}
by the composite of the following homomorphisms of complexes:
$$
\begin{matrix}
\underline{Hom}_{KVec_{\bold k}}(V^{\bullet},
A^{\bullet}\otimes W^{\bullet})\otimes
\underline{Hom}_{KVec_{\bold k}}(U^{\bullet},
A^{\bullet}\otimes V^{\bullet}) \\
\downarrow \\
\underline{Hom}_{KVec_{\bold k}}(A^{\bullet}\otimes V^{\bullet},
A^{\bullet}\otimes A^{\bullet}\otimes W^{\bullet})\otimes
\underline{Hom}_{KVec_{\bold k}}(U^{\bullet},
A^{\bullet}\otimes V^{\bullet}) \\
\downarrow \\
\underline{Hom}_{KVec_{\bold k}}(U^{\bullet},
A^{\bullet}\otimes A^{\bullet}\otimes W^{\bullet}) \\
\qquad \downarrow \mu\otimes 1 
\\
\underline{Hom}_{KVec_{\bold k}}(U^{\bullet},
A^{\bullet}\otimes W^{\bullet})
\end{matrix}
$$
\end{enumerate}
\end{definition}

\begin{remark}
Let $\Cal C$ be a DG category and $M$ be an object of  $\Cal C$.
Then $\underline{End}_{\Cal C}^{\bullet}(M)
=\underline{Hom}_{\Cal C}^{\bullet}(M,M)$
becomes an (associative) differential graded algebra.
We have $\underline{End}_{\Cal C_A}^{\bullet}(\bold k)\simeq (A^{\bullet})^{op}$
as DGA's.
Note that $(A^{\bullet})^{op}$ is a 
copy of $A^{\bullet}$ as a complex and the multiplication rule is
given by $a^{\circ}\cdot b^{\circ}=(-1)^{\deg(a)\deg(b)}(b\cdot a)^{\circ}$
\end{remark}

\subsection{DG complexes}
We introduce the notion of complexes in the setting of
DG category, which is called DG complexes.

\begin{definition}
\begin{enumerate}
\item
Let $\Cal C$ be a DC category.
A pair $(\{M^i\}_{i\in \bold Z},\{d_{ij}\}_{i>j})$
consisting of (1) a series of objects $\{M^i\}_{i\in\bold Z}$
 in $\Cal C$ indexed by $\bold Z$, 
and (2) a series of morphisms 
$$
d_{ij}\in \underline{Hom}_{\Cal C}^{j-i+1}(M^j,M^i)
$$
indexed by $i>j$ in $\bold Z$
is called a DG complex in $\Cal C$ if it
satisfies the following equality.
\begin{equation}
\label{condition for DG complex}
\partial(d_{ij})+\sum_{i>p>j}(-1)^{(i-p)(p-j+1)}d_{ip}\circ d_{pj}=0.
\end{equation}
The uncomfortable sign in this condition will be simplified by using
$d_{ij}^\#=d_{ij}\otimes t_{-i,-j}\in
\underline{Hom}_{\Cal C}^1(M^je^{-j},M^ie^{-i})$. Then the condition will be
\begin{equation}
\label{condition for DG complex twist}
\partial(d_{ij}^\#)+\sum_{i>p>j}d_{ip}^\#\circ d_{pj}^\#=0.
\end{equation}

\item
Let $M=(M^{\bullet}, d_M)$ and $N=(N^{\bullet}, d_N)$
be DG complexes in $\Cal C$.
We set
$$
\underline{Hom}_{K\Cal C}^i(M, N)=
\lim_{\underset{\alpha}\to}
\prod_{q-\alpha\leq r}\underline{Hom}_{\Cal C}^{i+q-r}(M^q,N^r).
$$
\item
Let $i \in \bold Z$.
For an element $\varphi\in \underline{Hom}_{K\Cal C}^i(M,N)$,
we define a map $D(\varphi) \in \underline{Hom}_{K\Cal C}^{i+1}(M, N)$
as follows. For an element
$$\varphi=(\varphi_{r,q}) \in
\prod_{q-\alpha\leq r}\underline{Hom}_{\Cal C}^{i+q-r}(M^q,N^r),
$$
we set 
$\varphi^\#_{rq}=\varphi_{rq}\otimes t_{-r,-q}\in
\underline{Hom}_{\Cal C}^i(M^qe^{-q},N^re^{-r})$ and
define
\begin{align*}
D(\varphi)_{r,q}^\#=& \partial(\varphi_{r,q}^\#) +\sum_{q-\alpha\leq r'<r}
d_{r,r'}^\#\circ \varphi_{r',q}^\# \\
&-(-1)^i\sum_{q<q'\leq r+\alpha}
\varphi_{r,q'}^\# \circ d_{q',q}^\# \\
&
\in \underline{Hom}_{\Cal C}^{i+1}(M^qe^{-q},N^re^{-r}) 
\end{align*}
This map defines a homomorphism 
$D:\underline{Hom}_{K\Cal C}^i(M, N)\to 
\underline{Hom}_{K\Cal C}^{i+1}(M, N)$, 
and the space $\underline{Hom}_{K\Cal C}^{\bullet}(M, N)$
becomes a complex of $\bold k$-vector spaces.
\item
Let $(L,d_L),(M,d_M)$ and $(N,d_N)$ be DG complexes.
We define the composite
$$
\begin{matrix}
\underline{Hom}_{K\Cal C}^{i}(M,N)\otimes
\underline{Hom}_{K\Cal C}^{j}(L,M) & \to &
\underline{Hom}_{K\Cal C}^{i+j}(L,N) \\
\psi\otimes \varphi & \mapsto & \psi\circ\varphi
\end{matrix}.
$$
by the relation
$$
(\psi\circ\varphi)_{p,q}^\#=\sum_{r}
\psi_{p,r}^\#\circ\varphi_{r,q}^\#\in\underline{Hom}^{i+j}_{\Cal C}
(L^qe^{-q},N^pe^{-p}).
$$
We can check that the the above composite is 
a well defined homomorphism of complexes
and associative.
\item
Let $M=(\{M^i\},\{d_{ij}\})$ be a DG complex.
If $M^i=0$ except for a finite number of $i$'s, 
it is called a bounded
DG complex.
\end{enumerate}
\end{definition}
It is easy to check the following proposition.
\begin{proposition}
Via the above differentials and the composites, 
the DG complexes in the
DG category $\Cal C$ becomes a DG category.  
\end{proposition}
\begin{definition}[DG cagegory of DG complexes]
The DG category of DG complexes defined as above is denoted as
$K\Cal C$.
The full sub DG category of bounded DG complexes in 
$K\Cal C$ is denoted as $K^b\Cal C$.
\end{definition}

\section{Topological motivation for simplicial bar complex}

\subsection{The space of marked path}
Before introducing simplicial bar complexes, we give a topological
motivation for Simplicial bar complexes. Let $X$ be a 
connected $C^{\infty}$ manifold with finite dimensional homology. 
For a sequence of integers 
$\alpha=(\alpha_0< \cdots <\alpha_n)\in \bold Z^{n+1}$, we define an 
$\alpha$-marking of $\bold R$ by
$$
t_{-,\alpha_0}\leq t_{\alpha_0,\alpha_1}\leq 
t_{\alpha_1,\alpha_2}\leq \cdots \leq 
t_{\alpha_{n-1},\alpha_n}\leq 
t_{\alpha_{n},+}
\in\bold R.
$$
For two indices $\alpha=(\alpha_0<\cdots <\alpha_n)$ and
$\beta=(\beta_0< \cdots < \beta_m)$, we write $\alpha < \beta$
if $\alpha \subset \beta$.
Let $\alpha, \beta$ be two indices such that $\alpha< \beta$.
For an $\alpha$-marking $\bold t$, we define a $\beta$-marking
$\bold u=i_{\alpha,\beta}(\bold t)$ by setting $u_{\beta_{j}\beta_{j+1}}=
t_{\alpha_i\alpha_{i+1}}$ where $\alpha_i\leq \beta_j$ and
$\beta_{j+1}\leq \alpha_{i+1}$.
For example, if $\alpha=(0,2),\beta=(0,1,2,3)$, then we have
$$
\begin{matrix}
\alpha_0=0 & & < & & \alpha_1=2 & & \\
\Vert & & & &\Vert &  & \\
\beta_0=0 &< &\beta_1=1 &< 
&\beta_2=2 &< &\beta_3=3
\end{matrix}
$$
and we have $u_{-,0}=t_{-,0}$, $u_{0,1}=u_{1,2}=t_{0,2}$,
 $u_{2,3}=u_{3,+}=t_{2,+}$.
The interval $[t_{\alpha_{i-1},\alpha_{i}},
t_{\alpha_{i},\alpha_{i+1}}]$
(it might be a one point)
is denoted as $I_{\alpha_i}(\bold t)$.
The intervals $(-\infty,t_{-,\alpha_0}]$ and
$[t_{\alpha_n,+},\infty)$ are denoted as 
$I_-(\bold t)$ and $I_+(\bold t)$,
respectively.

\setlength{\unitlength}{0.65mm}
\begin{picture}(200,50)(-20,-15)
\put(-20,10){\line(1,0){160}}
\put(10,8){\line(0,1){4}}
\put(45,8){\line(0,1){4}}
\put(105,8){\line(0,1){4}}
\put(-10, 0){$I_{-}(\bold u)$}
\put(20, 0){$I_{0}(\bold u)$}
\put(40, -8){$I_{1}(\bold u)$}
\put(60, 0){$I_{2}(\bold u)$}
\put(100, -8){$I_{3}(\bold u)$}
\put(120, 0){$I_{+}(\bold u)$}
\put(-10, 20){$I_{-}(\bold t)$}
\put(20, 20){$I_{0}(\bold t)$}
\put(60, 20){$I_{2}(\bold t)$}
\put(120, 20){$I_{+}(\bold t)$}
\put(45,0){\vector(0,1){10}}
\put(105,0){\vector(0,1){10}}
\put(0,14){$u_{-,0}=t_{-,0}$}
\put(40,14){$u_{0,1}=u_{1,2}=t_{0,2}$}
\put(100,14){$u_{2,3}=u_{3,+}=t_{2,+}$}
\end{picture}
\begin{center}
{\bf Figure 1. }
\end{center}
\vskip 0.2in

We define the set of $\alpha$-marked path
$P_{\alpha}(X)$ by
\begin{align*}
P_{\alpha}(X)=
\{(\gamma,\bold t)\mid &
\bold t \text{ is an $\alpha$-marking of }\bold R, \\
& \gamma:\bold R \to X \text{ is a path such that } \\
& \gamma(I_{-}(\bold t)), \gamma(I_{+}(\bold t))
\text{ are constants.}
\}.
\end{align*}
For a marked path $(\gamma,\bold t) \in P_{\alpha_0,\cdots ,\alpha_n}$,
$\gamma(I_{-}(\bold t))$ and
$\gamma(I_{+}(\bold t))$ are called the beginning point and the ending point,
respectively.
By evaluating the map at $t_{\alpha_i,\alpha_{i+1}}$, we have
an evaluation map $ev_{\alpha_0,\cdots ,\alpha_n}$ 
$$
P_{\alpha_0,\cdots ,\alpha_n}(X) \to X^{n+2}:(\gamma,\bold t)\mapsto
(\gamma(t_{-,\alpha_0}),
\gamma(t_{\alpha_0,\alpha_1}),\dots, \gamma(t_{\alpha_{n-1},\alpha_n}), 
\gamma(t_{\alpha_{n},+}))
$$
The product $X^{n+2}$ of $(n+2)$-copy of $X$ appeared in 
$ev_{\alpha_0,\cdots ,\alpha_n}$ is denoted as
$$
X_{\alpha_0,\dots, \alpha_n}=
X\overset{\alpha_0}\times X \overset{\alpha_1}\times \cdots
\overset{\alpha_{n-1}}\times X \overset{\alpha_{n}}\times X.
$$
The coordinates of $X_{\alpha_0,\dots,\alpha_n}$
are denoted as 
$$
(\tau_{-,\alpha_0},\tau_{\alpha_0,\alpha_1},
\dots, \tau_{\alpha_{n-1},\alpha_{n}}, 
\tau_{\alpha_{n},+}).
$$
We define $p_{\alpha}:X_{\alpha}\to X\times X$ by 
$\tau \mapsto (\tau_{-,\alpha_0},\tau_{\alpha_n,+})$.
We introduce a boundary maps $P_{\alpha}(X)\to P_{\beta}(X)$
for indices $\alpha < \beta$ using the map $i_{\alpha,\beta}$
defined as above. The map $f:X_{\alpha}\to X_{\beta}$
is defined by 
$(f(\tau))_{\beta_{j}\beta_{j+1}}=
\tau_{\alpha_i\alpha_{i+1}}$ by choosing $\alpha_i\leq \beta_j$ and
$\beta_{j+1}\leq \alpha_{i+1}$.
Note that $f$ is a composite of diagonal maps and it is compatible
with the maps $p_{\alpha},p_{\beta}$.
We have the following commutative diagram
for evaluation maps:
$$
\begin{matrix}
ev_{\alpha}:&P_{\alpha}(X) &\to & X_{\alpha} \\
& \downarrow &  & \downarrow \\
ev_{\beta}:&P_{\beta}(X) & \to& X_{\beta}
\end{matrix}
$$

\subsection{Chain complex associated to the system of marked paths}
The system $\{P_{\alpha}\}_{\alpha}$ and $\{X_{\alpha}\}_{\alpha}$
induces the following homomorphisms of double complexes
of singular chains by taking suitable summation with suitable sign.
$$
\begin{matrix}
\bold C_{\bullet}(P_{\bullet}):
\underset{\alpha_0}\prod C_{\bullet}(P_{\alpha_0}(X))& \to & 
\underset{\alpha_0<\alpha_1}\prod C_{\bullet}(P_{\alpha_0\alpha_1}(X)) & \to &
\underset{\alpha_0<\alpha_1<\alpha_2}\prod
C_{\bullet}(P_{\alpha_0\alpha_1\alpha_2}(X)) 
&\to\dots
\\
\downarrow & & \downarrow & &\downarrow
\\
\bold C_{\bullet}(X_{\bullet}):
\underset{\alpha_0}\prod C_{\bullet}(X_{\alpha_0})& \to & 
\underset{\alpha_0<\alpha_1}\prod C_{\bullet}(X_{\alpha_0\alpha_1}) & \to &
\underset{\alpha_0<\alpha_1<\alpha_2}\prod
C_{\bullet}(X_{\alpha_0\alpha_1\alpha_2}) 
&\to\dots
\end{matrix}
$$
Moreover, $\bold R$-valued
$C^{\infty}$ forms on $X_{\bullet}$ give rise to the following
double complex:
$$
\bold A^{\bullet}(X_{\bullet}):
\underset{\alpha_0}\oplus A^{\bullet}(X_{\alpha_0}) \ot 
\underset{\alpha_0<\alpha_1}\oplus A^{\bullet}(X_{\alpha_0,\alpha_1})  \ot 
\underset{\alpha_0<\alpha_1<\alpha_2}\oplus
A^{\bullet}(X_{\alpha_0,\alpha_1,\alpha_2}) 
\to\dots
$$

The following proposition is a consequence of 
Proposition \ref{acyclicity of free bar complex}.
\begin{proposition}
\label{acyclicity for topological setting}
\begin{enumerate}
\item
The associate simple complexes $C_{\bullet}(P_{\bullet})$ and
$C_{\bullet}(X_{\bullet})$ of
$\bold C_{\bullet}(P_{\bullet})$ and
$\bold C_{\bullet}(X_{\bullet})$ 
are acyclic.
\item
The associate simple complex  
$A^{\bullet}(X_{\bullet})$ of 
$\bold A^{\bullet}(X_{\bullet})$ is acyclic.
\end{enumerate}
\end{proposition}
\begin{remark}
Let $p,q$ be points of $X$, and $\epsilon_p,\epsilon_q:
A^{\bullet}(X)\to \bold R$ be augmentations obtained by
evaluation maps at $p$ and $q$, respectively.
Then the cohomology of 
$$
\bold R \otimes_{A^{\bullet},\epsilon_p}
A^{\bullet}(X_{\bullet})
\otimes_{A^{\bullet},\epsilon_q}\bold R
$$
is canonically isomorphic to the dual of the
nilpotent completion of $\bold R[\pi_1(X, p,q)]$,
where $\pi_1(X,p,q)$ is the set of homotopy classes of 
paths connecting $p$ and $q$.
\end{remark}

\section{Simplicial bar complex}
\subsection{Definition of simplicial bar complex}
Under the notation of the last section, 
$A^{\bullet}(X_{\alpha_0, \dots, \alpha_n})$
is quasi-isomorphic to 
${A^{\bullet}}^{\otimes (n+2)}(X)=
A^{\bullet}(X)\otimes_{\bold R}\cdots \otimes_{\bold R} A^{\bullet}(X)$
and the restriction map $A^{\bullet}(X_\beta)\to A^{\bullet}(X_\alpha)$
is compatible with a composite of multiplication
maps of $A$.
In this section, we define the simplicial bar complex of 
a differential graded algebra $A$
by imitating the definition of $A^{\bullet}(X_{\bullet})$.

Let $\bold k$ be a field and $A=A^{\bullet}$ be a DGA (associative but
might not be graded commutative) over $\bold k$.
We define the free simplicial bar complex $B( A )$ of $A$ as follows.
For a sequence of integers $\alpha=(\alpha_0<\cdots < \alpha_n)$,
we define $\bold B_{\alpha}(A)$ as a copy of  
$A^{\otimes (n+2)}$ which is written as
$$
\bold B_{\alpha}(A)=
A\overset{\alpha_0}\otimes A \overset{\alpha_1}\otimes\cdots 
\overset{\alpha_{n-1}}\otimes A \overset{\alpha_n}\otimes A.
$$
whose element will be written as
$
x_0 \overset{\alpha_0}\otimes x_1\overset{\alpha_1}\otimes \cdots 
\overset{\alpha_{n-1}}\otimes x_n
\overset{\alpha_{n}}\otimes x_{n+1}.
$
The length of the index $\alpha$ is defined by $n$ and denoted as 
$\mid\alpha\mid$. For an index $\beta=(\beta_0, \dots, \beta_{n-1})$,
we denoted $\beta < \alpha$ if $\beta\subset \alpha$.
We define a homomorphism $d_{\beta,\alpha}:B_\alpha (A) \to B_{\beta}(A)$
for
$\beta=(\alpha_0,\dots, \widehat{\alpha_i}, \dots, \alpha_n)$
by
$$
d_{\beta,\alpha}(
x_0 \overset{\alpha_0}\otimes \cdots 
\overset{\alpha_{i-1}}\otimes  
x_{i}\overset{\alpha_{i}}\otimes x_{i+1}
\overset{\alpha_{i+1}}\otimes \cdots 
\overset{\alpha_{n}}\otimes x_{n+1}
)=(-1)^{n-i}
x_0 \overset{\alpha_0}\otimes \cdots 
\overset{\alpha_{i-1}}\otimes  
x_{i}\cdot x_{i+1}
\overset{\alpha_{i+1}}\otimes \cdots 
\overset{\alpha_{n}}\otimes x_{n+1},
$$
where ``$\cdot$'' denotes the multiplication for the DGA $A^{\bullet}$.
We define a homomorphism of complexes 
$d_{\bold B}:\oplus_{\mid \alpha\mid =n}\bold B_{\alpha}(A) \to
\oplus_{\mid \beta\mid =n-1}\bold B_{\beta}(A)$ by
\begin{equation}
\label{simp bar outer diff}
d_{\bold B}=\sum_{\mid\alpha\mid=n,\mid
\beta\mid=n-1,\beta<\alpha}d_{\beta,\alpha}.
\end{equation}
Then $d^2=0$ and we have the following double complex $\bold B(A)$:
$$
\cdots \to \oplus_{\mid \alpha\mid =2}\bold B_{\alpha}(A)
\to \oplus_{\mid \alpha\mid =1}\bold B_{\alpha}(A)
\to \oplus_{\mid \alpha\mid =0}\bold B_{\alpha}(A) \to 0
$$
The free simplicial bar complex is defined as the
associate simple complex of the above double complex
and is denoted as $B(A)$.
The inner differential of $\bold B_{\alpha}(A)$ 
is denoted as $\delta$.

\begin{proposition}
\label{acyclicity of free bar complex}
The complex $B(A)$ is an acyclic complex.
\end{proposition}
\begin{proof}
Let $\bold B_{> N}(A)$ be the sub double complex of $\bold B(A)$
defined by
$$
\to \oplus_{N< \alpha_0<\alpha_1<\alpha_2}
\bold B_{\alpha_0,\alpha_1,\alpha_2}(A)
\to \oplus_{N< \alpha_0<\alpha_1}\bold B_{\alpha_0,\alpha_1}(A)
\to \oplus_{N< \alpha_0}\bold B_{\alpha_0}(A)
\to 0
$$
Then we have the natural inclusion $i:\bold B_{>N}(A)\to \bold B_{>N-1}(A)$.
We define a map $\theta:\bold B_{>N}(A)\to \bold B_{>N-1}(A)[-1]$ 
$$
\theta:\oplus_{N< \alpha_0<\cdots<\alpha_i}
\bold B_{\alpha_0,\cdots,\alpha_i}(A)
\to \oplus_{N-1< \beta<\alpha_0<\cdots<\alpha_i}
\bold B_{\beta,\alpha_0,\cdots,\alpha_i}(A)
$$
by 
\begin{align*}
& \theta(1\overset{\alpha_0}\otimes x_1\overset{\alpha_1}\otimes 
\cdots\overset{\alpha_{n-1}}\otimes x_n
\overset{\alpha_n}\otimes 1) 
= 
1\overset{N}\otimes 1\overset{\alpha_0}\otimes x_1\overset{\alpha_1}\otimes 
\cdots\overset{\alpha_{n-1}}\otimes x_n
\overset{\alpha_n}\otimes 1.
\end{align*}
Then we have $\theta\circ d + d\circ\theta=i(x)$. Therefore
the double complex (for the differential $d$)
$\bold B(A)=\underset{\underset{N}\to}\lim \bold B_{>N}(A)$ is 
an acyclic complex. Therefore its associate simple complex is also
acyclic.
\end{proof}
\begin{remark}
Let $X$ be a $C^{\infty}$ manifold. Let $A=A^{\bullet}(X)$
be the DGA of $C^{\infty}$ forms.
Since $X_{\alpha}$ is isomorphic
to the product of copies of $X$, the natural
homomorphism $\bold B_{\alpha}(A)\to A^{\bullet}(X_{\alpha})$
of complexes are quasi-isomorphism by K\"uneth formula.
The homomorphisms $\bold B_{\alpha}(A) \to \bold B_{\beta}(A)$
and $A^{\bullet}(X_{\alpha})\to A^{\bullet}(X_{\beta})$
are compatible with the above quasi-isomorphisms.
As a consequence, the natural homomorphism
$B(A) \to A^{\bullet}(X_{\bullet})$ is a quasi-isomorphism.
Proposition \ref{acyclicity for topological setting}
 follows from 
Proposition \ref{acyclicity of free bar complex}.
\end{remark}
We define left $A$ right $A$ action ($A$-$A$ action for short) 
on the complex $\bold B_{\alpha}(A)$ by
$$
\begin{matrix}
A\otimes \bold B_{\alpha}(A) \otimes A & \to & \bold B_{\alpha} \\
y\otimes (x_0 \overset{\alpha_0}\otimes x_1 \cdots 
x_n  
\overset{\alpha_{n}}\otimes x_{n+1})\otimes z & \to &
(yx_0) \overset{\alpha_0}\otimes x_1 \cdots 
 x_n 
\overset{\alpha_{n}}\otimes (x_{n+1} z).
\end{matrix}
$$ 
Since the differentials of $\bold B(A)$ are $A$-$A$ homomorphisms,
$\bold B(A)$ is an $A$-$A$ module.

We introduce a bar filtration $F_b^{\bullet}$ on $\bold B(A)$ as follows:
$$
F^{-i}_b\bold B(A):
\cdots \to 0 \to \oplus_{\mid \alpha\mid =i}\bold B_{\alpha}(A)
\to \cdots
\to \oplus_{\mid \alpha\mid =0}\bold B_{\alpha}(A) \to 0 
$$
Then we have the following spectral sequence
$$
E_1^{-i,p}=\oplus_{\mid \alpha \mid =i}
\bold B_{\alpha}(H^{\bullet}(A))^{p}\Rightarrow
E_{\infty}^{-i+p}=H^{-i+p}(B(A)).
$$
This spectral sequence is called the bar spectral sequence.

\subsection{Augmentations and coproduct}

Let $\epsilon_1,\epsilon_2$ be two augmentations of $A=A^{\bullet}$,
i.e. DGA homomorphisms $\epsilon_i:A\to \bold k$, where
$\bold k$ is the trivial DGA. We define 
$B(\epsilon_1\mid A \mid \epsilon_2)$ by the associate simple complex of
the double complex
$$
\bold B(\epsilon_1\mid A \mid \epsilon_2)
=\bold k\otimes_{A,\epsilon_1}\bold B(A)\otimes_{A,\epsilon_2}\bold k.
$$
Therefore the degree $-n$ part 
$\bold B(\epsilon_1 \mid A \mid \epsilon_2)_n$
of $\bold B(\epsilon_1 \mid A \mid \epsilon_2)$
is isomorphic to 
$\oplus_{\mid\alpha\mid=n}\bold B_{\alpha}(\epsilon_1\mid A\mid\epsilon_2)$
where
$$
\bold B_{\alpha}(\epsilon_1\mid A\mid\epsilon_2)
=\bold k\overset{\alpha_0}\otimes A \overset{\alpha_1}\otimes\cdots 
\overset{\alpha_{n-1}}\otimes A \overset{\alpha_n}\otimes \bold k
$$
and the differential (outer differential) is given by
\begin{align*}
d(1\overset{\alpha_0}\otimes x_1 \overset{\alpha_1}\otimes\cdots 
\overset{\alpha_{n-1}}\otimes x_n \overset{\alpha_n}\otimes 1) = &
\epsilon_1(x_1) \overset{\alpha_1}\otimes x_2 \overset{\alpha_2}\otimes\cdots 
\overset{\alpha_{n-1}}\otimes x_{n} \overset{\alpha_{n}}\otimes 1 \\
& + \sum_{i=1}^{n-1}(-1)^i
1\overset{\alpha_0}\otimes x_1 \cdots \overset{\alpha_{i-1}}\otimes 
x_ix_{i+1}
\overset{\alpha_{i+1}}\otimes\cdots x_n \overset{\alpha_n}\otimes 1 \\
& + (-1)^n
1 \overset{\alpha_0}\otimes x_1 \overset{\alpha_1}\otimes\cdots 
\overset{\alpha_{n-1}}\otimes x_{n-1} \overset{\alpha_{n-1}}\otimes 
\epsilon_2(x_n).
\end{align*}
As in the last subsection, we have the following bar spectral sequence
$$
E_1^{-i,p}=\oplus_{\mid \alpha \mid =i}
\bold B_{\alpha}(\epsilon_1 \mid H^{\bullet}(A)\mid \epsilon_2)^{p}\Rightarrow
E_{\infty}^{-i+p}=H^{-i+p}(B(\epsilon_1\mid A\mid \epsilon_2)).
$$

We introduce a coproduct structure on $B(A)$.
Let $\epsilon_1,\epsilon_2,\epsilon_3$ be augmentations on $A$.
Let $\alpha=(\alpha_0<\cdots <\alpha_n)$ be a sequence of integers.
We define $\Delta_{\epsilon_2}$
\begin{equation}
\label{coproduct}
\Delta_{\epsilon_2}:\bold B(A)\to 
(\bold B(A)\otimes_{\epsilon_2} \bold k)
\otimes
(\bold k\otimes_{\epsilon_2} \bold B(A))
\end{equation}
by
$$
\Delta_{\epsilon_2}(x_0 \overset{\alpha_0}\otimes x_1\overset{\alpha_1}
\otimes \cdots 
\overset{\alpha_{n-1}}\otimes x_n
\overset{\alpha_{n}}\otimes x_{n+1})
=\sum_{i=0}^n
(x_0 \overset{\alpha_0}\otimes \cdots\overset{\alpha_{i-1}}
\otimes x_i\overset{\alpha_{i}}\otimes 1)\otimes
(1\overset{\alpha_{i}}\otimes  x_{i+1}
\overset{\alpha_{i+1}}\otimes\cdots 
\overset{\alpha_{n}}\otimes x_{n+1})
$$
for an element in $\bold B_{\alpha}(A)$.
This is a morphism in the category $KKVect_{\bold k}$.
We introduce an $A-A$ module structure on the right hand side of
(\ref{coproduct}) by
\begin{align*}
&
y_1\cdot(x_0 \overset{\alpha_0}\otimes \cdots\overset{\alpha_{i-1}}
\otimes x_i\overset{\alpha_{i}}\otimes 1)\otimes
(1\overset{\alpha_{i}}\otimes  x_{i+1}
\overset{\alpha_{i+1}}\otimes\cdots 
\overset{\alpha_{n}}\otimes x_{n+1})\cdot y_2 \\
=& ((y_1x_0) \overset{\alpha_0}\otimes \cdots\overset{\alpha_{i-1}}
\otimes x_i\overset{\alpha_{i}}\otimes 1)\otimes
(1\overset{\alpha_{i}}\otimes  x_{i+1}
\overset{\alpha_{i+1}}\otimes\cdots 
\overset{\alpha_{n}}\otimes (x_{n+1} y_2))
\end{align*}
for $y_1,y_2 \in A$. We can prove the following by
direct computation.
\begin{proposition}
\label{properties of coproduct}
\begin{enumerate}
\item
\label{associativity for coalgebra}
The homomorphism $\Delta_{\epsilon_2}$ is a homomorphism of
complex of $KVect_{\bold k}$
and compatible with the $A$-$A$ action.
Therefore we have the following homomorphism in $KKVect_{\bold k}$:
$$
\Delta_{\epsilon_1,\epsilon_2,\epsilon_3}:
\bold B(\epsilon_1\mid A\mid \epsilon_3)\to 
\bold B(\epsilon_1 \mid A\mid \epsilon_2)
\otimes
\bold B(\epsilon_2 \mid A\mid \epsilon_3).
$$
Using the isomorphism (\ref{tensor of double complex to simple complex}),
we have the following morphism in
$KVect_{\bold k}$:
$$
\Delta_{\epsilon_1,\epsilon_2,\epsilon_3}:
B(\epsilon_1\mid A\mid \epsilon_3)\to 
B(\epsilon_1 \mid A\mid \epsilon_2)
\otimes
B(\epsilon_2 \mid A\mid \epsilon_3).
$$
\item
\label{counit of coagebra}
Let $\epsilon$ be an augmentation of $A$.
For $\alpha=(\alpha_0)$, let 
$\epsilon_{\alpha_0}:\bold B_{\alpha_0}=
\bold k\overset{\alpha_0}\otimes\bold k \to \bold k$ be the natural map.
Then $u=\sum_{\alpha_0}\epsilon_{\alpha_0}$ 
defines a homomorphism of 
differential graded coalgebras
$u:B(\epsilon \mid A\mid \epsilon) \to \bold k$.
\item
The following composite maps are identity:
\begin{align}
\label{counit properties}
B(\epsilon_1\mid A\mid \epsilon_2)&\overset{
\Delta_{\epsilon_1,\epsilon_1,\epsilon_2}}
\to 
B(\epsilon_1 \mid A\mid \epsilon_1)
\otimes
B(\epsilon_1 \mid A\mid \epsilon_2) \\
\nonumber
&\overset{u\otimes 1}\to B(\epsilon_1 \mid A\mid \epsilon_2) \\
\nonumber
B(\epsilon_1\mid A\mid \epsilon_2)&\overset{
\Delta_{\epsilon_1,\epsilon_2,\epsilon_2}}
\to 
B(\epsilon_1 \mid A\mid \epsilon_2)
\otimes
B(\epsilon_2 \mid A\mid \epsilon_2) \\
\nonumber
&\overset{1\otimes u}\to B(\epsilon_1 \mid A\mid \epsilon_2) 
\end{align}
where $\Delta_{\epsilon_1,\epsilon_2,\epsilon_3}$
is defined in (\ref{associativity for coalgebra}).
\end{enumerate}
\end{proposition}
\begin{definition}
\begin{enumerate}
\item
The map $\Delta_{\epsilon_1,\epsilon_2,\epsilon_3}$ 
defined in (\ref{associativity for coalgebra}) of 
Proposition \ref{properties of coproduct} is
called the coproduct of $B(A)$. It is easy to see 
that the coproduct is coassociative. Thus $B(A)$
forms a DG coalgebroid over the set $Spaug(A)$ of augmentations of 
$A$.
\item
The map $u:B(\epsilon \mid A \mid \epsilon)\to \bold k$ 
defined in (\ref{counit of coagebra}) of 
Proposition \ref{properties of coproduct} is
is called 
the counit. In general, for an associative differential graded 
coalgebra $B$, the map $u$ satisfying the properties
as (\ref{counit properties}) is called a counit.
\end{enumerate}
\end{definition}

\section{Comparison to Chen's theory}

In this section, we compare the simplicial bar complex defined
as above and the reduced bar complex defined by Chen.
Let $A=A^{\bullet}$ be a differential graded algebra over a field 
$\bold k$ and $\epsilon:A \to \bold k$ be an augmentation.
Let $I=\Ker(\epsilon:A \to \bold k)$ be the augmentation ideal.
We define the double complex $\bold B_{red}(A,\epsilon)$ as
$$
\bold B_{red}(A,\epsilon): \cdots \to I^{\otimes 3} ]
\overset{d_{B,red}}\to  I^{\otimes 2}
\overset{d_{B,red}}\to I \overset{0}\to \bold k \to 0,
$$
where the outer differential 
$d_{B,red}:I^{\otimes i}\to I^{\otimes (i-1)}$ is defined by
\begin{equation}
\label{reduced outer}
d_{B,red}:[x_1\mid \cdots \mid x_i]\mapsto 
\sum_{p=1}^{i-1}(-1)^{p}
[x_1\mid \cdots \mid x_{p}x_{p+1}\mid\cdots 
\mid x_i].
\end{equation}
Here an element $x_1\otimes \cdots \otimes x_i$ in $I^{\otimes i}$
is denoted as $[x_1\mid \cdots \mid x_i]$.
The total complex $B_{red}(A,\epsilon)$ is called the
Chen's reduced bar complex.

We define a degree preserving linear map  
$B(\epsilon \mid A \mid \epsilon)\to B_{red}(A,\epsilon)$
by
\begin{equation}
\label{simp to chen}
1\overset{\alpha_0}\otimes x_1\overset{\alpha_1}\otimes \cdots 
\overset{\alpha_{n-1}}\otimes x_{n-1}\overset{\alpha_n}\otimes 1
\mapsto
[\pi(x_1)\mid \cdots 
\mid\pi(x_{n-1})]
\end{equation}
where $\pi(x)=x-\epsilon(x)\in I$ for an element $x \in A$.
\begin{lemma}
It is a homomorphism of complexes of $KVect_{\bold k}$.
\end{lemma}
We have the following theorem.
\begin{theorem}
\label{comparison of simp and chen}
The homomorphism of associate simple complexes
$B(\epsilon \mid A \mid \epsilon)\to B_{red}(A,\epsilon)$
obtained form the map (\ref{simp to chen}) is a quasi-isomorphism.
\end{theorem}
\begin{proof}
First, we define the bar filtration $F^{-i}_b\bold B_{red}(A,\epsilon)$
on the double complex
$\bold B_{red}(A, \epsilon)$ by
$$
F^{-i}_b\bold B_{red}(A,\epsilon):
\cdots \to 0 \to I^{\otimes i}
\to \cdots
\to I \to \bold k \to 0. 
$$

Then the bar filtration on $B(\epsilon\mid A\mid\epsilon)$
and that on $B_{red}(A,\epsilon)$ are compatible and we have
the following homomorphisms of spectral sequences:
$$
\begin{matrix}
E_1^{-i,p}=\oplus_{\mid \alpha\mid =i}
\bold B_{\alpha}(\epsilon \mid H^{\bullet}(A)\mid \epsilon)^p & \Rightarrow & 
E_{\infty}^{-i+p}=H^{-i+p}(B(\epsilon\mid A \mid\epsilon)) \\
\downarrow & & \downarrow
\\
'E_1^{-i,p}=(H^{\bullet}(I)^{\otimes i})^p & \Rightarrow & 
'E_{\infty}^{-i+p}=H^{-i+p}(B_{red}(A, \epsilon))
\end{matrix}
$$
We prove that the vertical arrow coincides from $E_2$-terms.
$E_1$-terms are the following complexes of graded vector spaces:
$$
\begin{matrix}
\cdots
&\to&
\underset{\alpha_0<\alpha_1<\alpha_2}\oplus
B_{\alpha_0,\alpha_1,\alpha_2}
&\to&
\underset{\alpha_0<\alpha_1}\oplus
B_{\alpha_0,\alpha_1}
&\to&
\oplus_{\alpha_0}
B_{\alpha_0}
&\to& 0 \\
& &\downarrow & & \downarrow & &\downarrow & &\\
\cdots
&\to&
H^{\bullet}(I)^{\otimes 2}
&\to&
H^{\bullet}(I)^{\otimes 1}
&\to&
\bold k
&\to& 0 \\
\end{matrix}
$$
Here we used the abbreviation
$$
B_{\alpha_0,\cdots,\alpha_n}=
\bold k\overset{\alpha_0}\otimes H^{\bullet}(A)
\otimes \cdots
\otimes 
H^{\bullet}(A)
\overset{\alpha_n}\otimes \bold k.
$$
We show that the vertical homomorphism of complexes are quasi-isomorphism.
We have
\begin{align*}
B_{\alpha_0,\cdots,\alpha_n}& =
\bold k\overset{\alpha_0}\otimes (H^{\bullet}(I)\oplus \bold k)
\otimes \cdots
\otimes 
(H^{\bullet}(I)\oplus \bold k)
\overset{\alpha_n}\otimes \bold k \\
& =
\oplus_{k=0}^{n}\oplus_{c\in C(n,k)}
H^{\bullet}(I)^{\otimes k} ,
\end{align*}
where $C(n,k)$ is the subset of 
$\{1,\dots,n\}$ of cardinality $k$.
Since the images of 
$H^{\bullet}(I)\otimes H^{\bullet}(A)$ and
$H^{\bullet}(A)\otimes H^{\bullet }(I)$ under the multiplication
is contained in
$H^{\bullet}(I)$, the following filtration $G$ on
$(E_1^{-i,\bullet},d_1)_i$ and
$('E_1^{-i,\bullet},d_1)_i$ defines a subcomplexes:
\begin{align*}
G^{-l}E_1^{-i,\bullet}&=
\oplus_{k=0}^{l}\oplus_{c\in C(n,k)}
H^{\bullet}(I)^{\otimes k}, \\
G^{-l}\ 'E_1^{-i,\bullet}
&=\begin{cases}
H^{\bullet}(I)^{\otimes i}  &\quad (i \leq l)  \\ 
0  &\quad (i > l)  
\end{cases}
\end{align*}
Thus the associate graded complexes of $E_1^{-i,\bullet}\to
'E_1^{-i,\bullet}$ are
\begin{equation}
\label{graded comparison}
\begin{matrix}
\cdots
&\to&
\underset{\substack{\mid \alpha \mid =l+1,\\ c\in C(l+1,l)}}\oplus
H^{\bullet}(I)^{\otimes l}
&\to&
\underset{\substack{\mid \alpha \mid =l, \\ c\in C(l,l)}}\oplus
H^{\bullet}(I)^{\otimes l}
&\to&
0
&\to& \cdots \\
& &\downarrow & & \downarrow & &\downarrow & &\\
\cdots
&\to&
0
&\to&
H^{\bullet}(I)^{\otimes l}
&\to&
0
&\to& \cdots \\
\end{matrix}
\end{equation}
We consider an algebra $\bold k \oplus \bold k x$ such that $x^2=0$.
We define the complex $K_{m,l}$ by
\begin{equation}
\label{comparison typical basic}
\cdots
\to
\underset{\substack{\mid \alpha \mid =l+2,\\ 
\alpha\subset [1,m] \\
c\in C(l+2,l)}}\oplus
\bold k t_{\alpha,c}
\to
\underset{\substack{\mid \alpha \mid =l+1,\\ 
\alpha\subset [1,m] \\
c\in C(l+1,l)}}\oplus
\bold kt_{\alpha,c}
\to
\underset{\substack{\mid \alpha \mid =l,\\ 
\alpha\subset [1,m] \\
c\in C(l,l)}}\oplus
\bold kt_{\alpha,c}
\to
0
\end{equation}
where 
\begin{align*}
&t_{\alpha,c}=1\overset{\alpha_0}\otimes u_1\overset{\alpha_1}
\otimes\cdots\overset{\alpha_{k-1}}\otimes u_k\overset{\alpha_k}\otimes 1, \\
&\text{ with }u_i=\begin{cases}x &\text{ if }i\in c, \\
1 &\text{ if }i\notin c.
\end{cases}
\end{align*}
The differential is similar to that of simplicial bar complex.
Then there is a natural inclusion $K_{m,l}\to K_{m+1,l}$
and a map $\epsilon_m:K_{m,l}\to \bold k$.
\begin{proposition}
\label{quasi-iso for typical basic}
\begin{enumerate}
\item
The map $\epsilon_l:K_{l,l}\to \bold k$ is an isomorphism.
\item
The natural inclusion $K_{m,l}\to K_{m+1,l}$ is a quasi-isomorphism
for $m\geq l$.
\item
The complex $K_{m,l}\to \bold k$ is an acyclic complex for $l\leq m$.
\end{enumerate}
\end{proposition}
\begin{proof}
We prove the proposition by the induction on $l$.
The statement (1) is direct from the definitions.
The statement (3) follows from the statement (1) and (2).
We prove the statement (2) for $l$ assuming the statement (3) for $l-1$.
The cokernel of the complex $\Coker(K_{m,l}\to K_{m+1,l})$ is isomorphic
to
$$
\cdots
\to
\underset{\substack{\mid \alpha \mid =k+1,\\ 
\alpha\subset [1,m+1],\\
\alpha_{k+1}=m+1}}\oplus \\
\left(
\begin{matrix}
\underset{c\in C(k,l)}\oplus
\bold k t'_{\alpha,c}\\
\bigoplus \\
\underset{c\in C(k,l-1)}\oplus
\bold k t''_{\alpha,c}\\
\end{matrix}
\right)
\overset{d}\to
\underset{\substack{\mid \alpha \mid =k,\\ 
\alpha\subset [1,m+1],\\
\alpha_{k}=m+1}}\oplus \\
\left(
\begin{matrix}
\underset{c\in C(k-1,l)}\oplus
\bold k t'_{\alpha,c}\\
\bigoplus \\
\underset{c\in C(k-1,l-1)}\oplus
\bold k t''_{\alpha,c}\\
\end{matrix}
\right)
\to\cdots
$$
where
\begin{align*}
&t'_{\alpha,c}=1\overset{\alpha_0}\otimes u_1\overset{\alpha_1}
\otimes\cdots\overset{\alpha_{k-1}}\otimes u_k
\overset{\alpha_k}\otimes 1
\overset{m+1}\otimes 1, \\
&t''_{\alpha,c}=1\overset{\alpha_0}\otimes u_1\overset{\alpha_1}
\otimes\cdots\overset{\alpha_{k-1}}\otimes u_k
\overset{\alpha_k}\otimes x
\overset{m+1}\otimes 1, \\
&\text{ with }u_i=\begin{cases}x &\text{ if }i\in c, \\
1 &\text{ if }i\notin c
\end{cases}
\end{align*}
Thus it is isomorphic to the cone of a complex homomorphism
$K_{m,l}\to K_{m,l-1}$ by considering the both cases for $u_k=1,x$.
Since $K_{m,l}\to \bold k$ and $K_{m,l-1}\to \bold k$ is
quasi-isomorphism by the induction hypothesis for $(m,l)$
and $(m,l-1)$, we have the statement (2).
\end{proof}
Proof of Theorem \ref{comparison of simp and chen}. 
Since the diagram (\ref{graded comparison})
 is obtained from (\ref{comparison typical basic})
 by tensoring 
$H^{\bullet}(I)^{\otimes l}$ and
taking the inductive limit on $m$,
the homomorphism of associate graded complex of
$E_1^{-i,\bullet}\to
'E_1^{-i,\bullet}$ is a
quasi-isomorphism by Proposition \ref{quasi-iso for typical basic}.
This proves the theorem.
\end{proof}
\begin{remark}
Since the spectral sequence associated to the filtration $G$ is
isomorphic from $E_2$ terms, the induced filtration on
$H^{i}(B(\epsilon \mid A \mid \epsilon))$ and
that on 
$H^{i}(B_{red}(A,\epsilon))$ are equal.
\end{remark}

\section{Simplicial bar complex and DG category $K\Cal C_A$}

\label{main section}

\subsection{Definition of comodules over DG coalgebras}

In this section, we prove  that the DG category of
comodules over the simplicial bar complex 
$B(\epsilon \mid A\mid\epsilon)$
is homotopy equivalent to the DG category $K\Cal C_A$.

Let $B$ be a coassociative differential graded coalgebra over 
$\bold k$ with a counit 
$u:B\to \bold k$. The comultiplication $B\to B\otimes B$ is written as
$\Delta_B$.
\begin{definition}[DG categroy of $B$ comodules]
\label{def of B comod as DG cat}
\begin{enumerate}
\item
A $\bold k$-complex $M$ with a homomorphism $\Delta_M$ of complexes
$$
\Delta_M:M \to B\otimes M
$$
is called a (left) $B$-comodule if the following properties hold.
\begin{enumerate}
\item {\sc Coassociativity}
The following diagram commutes:
$$
\begin{matrix}
M & \overset{\Delta_M}\to & B\otimes M \\
\Delta_M\downarrow & & \downarrow 1\otimes \Delta_M \\
B\otimes M &\overset{\Delta_B\otimes 1}\to & B\otimes B\otimes M.
\end{matrix}
$$
\item {\sc Counitarity}
The composite homomorphism
$
M \overset{\Delta_M}\to B\otimes M \overset{u\otimes 1}\to M
$
is the identity of $M$.
\end{enumerate}
\item
Let $M,N$ be $B$-comodules. 
We define the complex of homomorphisms
$\underline{Hom}^{\bullet}_{B-com}(M,N)$ by the associate simple complex
of $\underline{{\bf Hom}}_{B-com}^{\bullet}(M,N)$ defined by
\begin{align*}
\underline{{\bf Hom}}^{\bullet}_{B-com}(M):
& \underline{Hom}_{KVec_{\bold k}}(M,N)
\overset{d_H}\to
\underline{Hom}_{KVec_{\bold k}}(M,B\otimes N) 
\\
&\overset{d_H}\to 
\underline{Hom}_{KVec_{\bold k}}(M,B\otimes B\otimes N)
\overset{d_H}\to\dots 
\end{align*}
where 
$$
\underline{Hom}_{KVec_{\bold k}}(M,B^{\otimes n}\otimes N) \\
\overset{d_H}\to 
\underline{Hom}_{KVec_{\bold k}}(M,B^{\otimes (n+1)}\otimes N)
$$
is defined by 
\begin{align}
d_H\varphi=&(-1)^{n+1}(1_B\otimes \varphi)\circ\Delta_M 
\label{outer hochschild}
\\
\nonumber
&+\sum_{i=1}^n(-1)^{n-i+1}
(1^{\otimes (i-1)}_B\otimes \Delta_B\otimes 1^{\otimes (n-i)}_B
\otimes 1_N)\circ
\varphi \\
\nonumber
&+(1^{\otimes n}_B\otimes \Delta_N)\circ \varphi.
\end{align}
\item
Let $L,M, N$ be $B$-comodules.
We define the composite morphism
$$
\mu:\underline{Hom}_{B-com}(M,N)\otimes
\underline{Hom}_{B-com}(L,M) \to
\underline{Hom}_{B-com}(L,N)
$$
by $\sum_{i,j\geq 0}\mu_{ij}$, where the morphism
\begin{align*}
\mu_{ij}:&\underline{Hom}_{KVec_{\bold k}}(M,B^{\otimes i}\otimes
     N)e^{-i}\otimes
\underline{Hom}_{KVec_{\bold k}}(L,B^{\otimes j}\otimes M)e^{-j} \\
& \to
\underline{Hom}_{KVec_{\bold k}}(L,B^{\otimes (i+j)}\otimes M)e^{-i-j}
\end{align*}
is defined by $\mu_{i,j}(f\otimes g)=
(1_B^{\otimes j}\otimes f)\circ g$.
\end{enumerate}
\end{definition}
\begin{proposition}
The category of $B$-comodules $(B-com)$ forms
a DG-category by setting
\begin{enumerate}
\item
the complex of homomorphism by
$\underline{Hom}_{B-com}^{\bullet}(\bullet,\bullet)$, and 
\item
the composite homomorphism by $\mu$. The composite is denoted
as ``$\circ$''.
\end{enumerate}
\end{proposition}
\begin{proof}
To show that the multiplication homomorphism
$\mu$ is a homomorphism of complex, it is enough to consider
the outer differentials.
Let 
$$
f\otimes g\in 
\underline{Hom}_{KVec_{\bold k}}^\bullet(M,B^{\otimes i}\otimes
     N)\otimes
\underline{Hom}_{KVec_{\bold k}}^\bullet(L,B^{\otimes j}\otimes M)
$$
and $d_H$ be the outer differential.
Then we have
\begin{align*}
d_H(f\circ g)
=&(-1)^{n+1}(1_B^{\otimes (j+1)}\otimes f)\circ (1_B\otimes g)\circ\Delta_L
\\
&+\sum_{p=1}^n(-1)^{n-p+1}
(1^{\otimes (p-1)}_B\otimes \Delta_B\otimes 1^{\otimes (n-p)}_B
\otimes 1_N)\circ
(1_B^{\otimes j}\otimes f)\circ g
\\
&+(1^{\otimes n}_B\otimes \Delta_N)\circ (1_B^{\otimes j}\otimes f)\circ g
\\
=&(-1)^{i+j+1}(1_B^{\otimes (j+1)}\otimes f)\circ (1_B\otimes g)\circ\Delta_L
\\
&+(-1)^i\sum_{p=1}^j(-1)^{j-p+1}
(1_B^{\otimes (j+1)}\otimes f)\circ
(1^{\otimes (p-1)}_B\otimes \Delta_B\otimes 1^{\otimes (j-p)}_B
\otimes 1_N)\circ g
\\
&+(-1)^i(1_B^{\otimes (j+1)}\otimes f)\circ 
(1_B^{\otimes j}\otimes \Delta_M)\circ
g
\\
&+(-1)^{i+1}(1_B^{\otimes (j+1)}\otimes f)\circ 
(1_B^{\otimes j}\otimes \Delta_M)\circ
g
\\
&+\sum_{q=1}^i(-1)^{i-q+1}
(1^{\otimes (q+j-1)}_B\otimes \Delta_B\otimes 1^{\otimes (i-q)}_B
\otimes 1_N)\circ
(1_B^{\otimes j}\otimes f)\circ g
\\
&+(1^{\otimes (i+j)}_B\otimes \Delta_N)\circ (1_B^{\otimes j}\otimes f)\circ g
\\
=&
(-1)^if\circ d_H(g)+
d_H(f)\circ g.
\end{align*}
The associativity for the composite can be proved similarly.
\end{proof}
Let $A$ be a DGA and $\epsilon:A\to \bold k$ be
an augmentation. Until the end of this subsection,
let $B=B(\epsilon\mid A\mid \epsilon)$
be the simplicial bar complex 
and $B_{red}=B_{red}(A,\epsilon)$ be the reduced bar complex
for the augmentation $\epsilon$.
Let $\pi_\alpha:B\to \bold k\overset{\alpha}\otimes \bold k$
be the projection.
\begin{definition}[$(B-com)^{red,b}$]
Let $S\subset \bold Z$ a subset of $\bold Z$.
A $B$-comodule
$M$ is said to be supported on $S$ if and only if the composite
$$
M\overset{\Delta_M}\to B\otimes M 
\overset{\pi_\alpha \otimes 1_M}\to \bold k
\overset{\alpha}\otimes\bold k\otimes M
$$
is zero if $\alpha \notin S$. A module $M$ supported on a finite
set $S$ is called a bounded $B$ comodule.
The class of objects in the DG category $(B-com)^{red,b}$ is defined by
bounded $B$ comodules and the complex of morphisms from $M$ to $N$ 
is defined by $\underline{Hom}_{B_{red}-com}(M_{red},N_{red})$.
Here $M_{red}$ and $N_{red}$ are the 
$B_{red}$-comodules induced by $M$ and $N$.
\end{definition}

The rest of this section is spent to prove the following
theorem.
\begin{theorem}[Main Theorem]
\label{first main theorem}
DG categories $K^b\Cal C_A$ and $(B-com)^{red,b}$ are homotopy equivalent.
\end{theorem}

\subsection{Correspondences on objects}

We construct a one to one correspondence $\varphi$ from
the class of objects of $(K^b\Cal C_A)$ to that of objects
of $(B-com)^{red,b}$.
In this section, we simply denote $\bold B_{\alpha}$ and $B$
for $\bold B_{\alpha}(\epsilon\mid A\mid \epsilon)$ and $B(\epsilon\mid A\mid \epsilon)$,
respectively.

\subsubsection{Definition of $\varphi:ob(K^b\Cal C_A)\to 
ob(B-com)^{red,b}$}
Let $M=(M^{i},d_{ij})$ be an object of $K^b\Cal C_A$, where 
$M^i \in \Cal C_A$.
We set $s(M)$ by the graded vector space
$\oplus_{i}M^{i}e^{-i}$.
The morphism 
$d_{ij}\in \underline{Hom}_{KVec_{\bold k}}^{j-i+1}
(M^j, A^{\bullet}\otimes M^i)$ defines an element
\begin{align}
\bold D_{ji}=d_{ij}\otimes t_{-i,-j}\in &
\underline{Hom}_{KVec_{\bold k}}^1
(M^{j}e^{-j},A^{\bullet}\otimes M^{i}e^{-i}) \\
\nonumber
=&
\underline{Hom}_{KVec_{\bold k}}^1
(M^{j}e^{-j},
(\bold k\overset{j}\otimes A^{\bullet}
\overset{i}\otimes \bold k) \otimes
 M^{i}e^{-i}) \\
\nonumber
=&
\underline{Hom}_{KVec_{\bold k}}^1
(M^{j}e^{-j},
\bold B_{ji}
\otimes
 M^{i}e^{-i}). 
\end{align}
We consider the following linear map
$$
M^ke^{-k} \overset{\bold D_{kj}}\to \bold B_{kj}\otimes M^je^{-j}
\overset{1\otimes \bold D_{ji}}\to \bold B_{kji}\otimes M^ie^{-i}
\overset{\mu\otimes 1}\to \bold B_{ki}\otimes M^ie^{-i},
$$
where $\mu$ is the multiplication map.
The condition (\ref{condition for DG complex}) is equivalent to the relation
\begin{align}
\label{another expression for DG condition}
& \sum_{j:k<j<i}(\mu \otimes 1)(1\otimes \bold D_{ji})\bold D_{kj} \\
\nonumber
+&(d_A\otimes 1+1\otimes d_{M^i})\bold D_{ki}+
\bold D_{ki} d_{M^{k}}=0
\end{align}
in $\underline{Hom}_{KVec_{\bold k}}^2
(M^{k}e^{-k},\bold B_{ki}\otimes M^{i}e^{-i})$.

We set $s(M)=\oplus_i M^ie^{-i}$.
Then the sum $\bold D=\sum_{j<i}\bold D_{j,i}$ defines an element
in 
$
Hom_{\Cal C_{\bold k}}^1(s(M),A^{\bullet}\otimes s(M)).
$
By composing $\epsilon\otimes 1\in
Hom_{\Cal C_{\bold k}}^0(A^{\bullet}\otimes s(M), s(M)).
$
we have
\begin{equation}
\label{diffrential for N}
D_{\epsilon}=(\epsilon \otimes 1)\bold D
\in \underline{Hom}_{KVec_{\bold k}}^1(s(M), s(M)).
\end{equation}
\begin{lemma}
We set $d_M=\sum_i d_{M^i}$.
Then 
$
\delta_M=d_{M}+D_{\epsilon}
$ 
defines a differential 
on $s(M)=\oplus_i M^ie^{-i}$.
\end{lemma}
\begin{proof}
By (\ref{another expression for DG condition}),
we have
\begin{align*}
 (\mu \otimes 1)(1\otimes \bold D)\bold D 
+(d_A\otimes 1+1\otimes d_{M})\bold D+
\bold D d_{M}=0
\end{align*}
By the commutative diagram
$$
\begin{matrix}
s(M) &\overset{\bold D}\to &
A^{\bullet}\otimes s(M)
&\overset{\epsilon\otimes 1}\to 
& s(M) \\
& & \downarrow  1\otimes \bold D
& & \downarrow \bold D\\
& & A^{\bullet}\otimes A^{\bullet}\otimes s(M)
& \overset{\epsilon\otimes 1 \otimes 1}\to &
A^{\bullet}\otimes s(M) \\
& & \downarrow  \mu \otimes 1
& & \downarrow \epsilon\otimes 1\\
& & A^{\bullet}\otimes s(M)
& \overset{\epsilon\otimes 1}\to &
s(M),
\end{matrix}
$$
we have $(\epsilon\otimes 1) (\mu \otimes 1)(1\otimes \bold D)\bold D
 =D_{\epsilon}^2$, and by
$$
\begin{matrix}
A\otimes s(M) &\overset{d_A\otimes 1+ 1\otimes d_M}\to &A\otimes s(M) \\
\epsilon \otimes 1\downarrow & & \downarrow \epsilon\otimes 1 \\
s(M) &\underset{d_M}\to &s(M),
\end{matrix}
$$
we have $(\epsilon\otimes 1)
(d_A\otimes 1+ 1\otimes d_M)=d_M(\epsilon\otimes 1)$.
Therefore we have $(d_M+D_{\epsilon})^2=0$
\end{proof}
\begin{definition}
Let $\alpha=(\alpha_0<\cdots<\alpha_n)$ be a sequence of integers.
We define 
$\bold D_{\alpha}=\bold D_{\alpha_0, \dots, \alpha_n}\in 
\underline{Hom}_{\Cal C_{\bold k}}^n(M^{\alpha_0}e^{-\alpha_0},
\bold B_{\alpha}\otimes M^{\alpha_n}e^{-\alpha_n})$
inductively by the following composite homomorphism:
\begin{align*}
M^{\alpha_0}e^{-\alpha_0} &
\overset{\bold D_{\alpha_0, \dots, \alpha_{n-1}}}
\longrightarrow \bold B_{\alpha_0, \dots, \alpha_{n-1}}
\otimes M^{\alpha_{n-1}}e^{-\alpha_{n-1}} \\
&
\overset{1\otimes \bold D_{\alpha_{n-1},\alpha_{n}}}
\longrightarrow
(\bold B_{\alpha_0, \dots, \alpha_{n-1}}
\overset{\alpha_{n-1}}\otimes A^{\bullet})
\otimes M^{\alpha_{n}}e^{-\alpha_n} =
\bold B_{\alpha_0, \dots, \alpha_{n}}
\otimes M^{\alpha_{n}}e^{-\alpha_n}.
\end{align*}
By the isomorphism
$$
\begin{matrix}
\underline{Hom}_{\Cal C_{\bold k}}^n(M^{\alpha_0}e^{-\alpha_0},
\bold B_{\alpha}\otimes M^{\alpha_n}e^{-\alpha_n}) &\to &
\underline{Hom}_{\Cal C_{\bold k}}^0(M^{\alpha_0}e^{-\alpha_0},
\bold B_{\alpha}e^n\otimes M^{\alpha_n}e^{-\alpha_n}) \\
\varphi &\mapsto &(1_A^{\otimes n}\otimes t^{-n}\otimes 1)\varphi
\end{matrix}
$$
the element corresponding to $\bold D_{\alpha}$ is denoted by $D_{\alpha}$.
We set 
$$\Delta=\sum_{\alpha}D_\alpha:s(M)\to B\otimes s(M).$$
\end{definition}
\begin{proposition}
\begin{enumerate}
\item
The map $\Delta$ is a homomorphism of complexes.
\item
The above homomorphism $\Delta$ defines a $B$-comodule structure
on $s(M)$.
\end{enumerate}
\end{proposition}
\begin{proof}
For the proof of (2), the coassociativity 
and counitarity is easy to check, so we omit the proof.
We prove the statement (1).
We compute the right hand side of the following equality:
\begin{align}
\nonumber
&dD_{\alpha_0,\dots, \alpha_n}(x) \\
\label{1-1}
=
& (d_{\bold B}\otimes t\otimes 1_M)D_{\alpha_0,\dots, \alpha_n}(x)\\
\label{1-2}
&+\sum_{i=0}^{n-1}(1_A^{\otimes i}\otimes d_A\otimes 1_A^{\otimes (n-i-1)}\otimes 1_M)
D_{\alpha_0,\dots, \alpha_n}(x) \\
\label{1-3}
&+(1_A^{\otimes n}\otimes d_{s(M)})D_{\alpha_0,\dots, \alpha_n}(x).
\end{align}
Here $d_{\bold B}$ is the outer differential defined in 
(\ref{simp bar outer diff}).
Then the term (\ref{1-1}) is equal to
\begin{align}
\nonumber
& (d_{\bold B}\otimes t \otimes 1_M)
D_{\alpha_0,\dots, \alpha_n}(x)= \\
\label{2-1}
&
(1^{\otimes (n-1)}_A\otimes t\otimes 1_M)(\epsilon\otimes 1_A^{\otimes (n-1)}\otimes 1_M)D_{\alpha_0,\dots, \alpha_n}(x) \\
\label{2-2}
&+
(1^{\otimes {n-1}}\otimes t\otimes 1_M)\sum_{i=1}^{n-1}(-1)^{i}
(1^{\otimes (i-1)} \otimes \mu  
\otimes 1^{\otimes (n-i-1)}\otimes 1_M)
D_{\alpha_0,\dots, \alpha_n}(x) \\
\label{2-3}
&+(-1)^{n}
(1^{\otimes (n-1)}\otimes t\otimes 1_M)(1^{\otimes (n-1)}  \otimes \epsilon\otimes 1_M)
D_{\alpha_0,\dots, \alpha_n}(x).
\end{align}
We define a map $$
\bold E(\alpha_1,\dots \alpha_{i-1};\alpha_i;
\alpha_{i+1},\dots ,\alpha_n)\in
\underline{Hom}_{\Cal C_{\bold k}}^{n+1}
(M^{\alpha_0}e^{-\alpha_0},\bold B_{\alpha_0,\dots, \alpha_n}\otimes
 M^{\alpha_n}e^{-\alpha_n})
$$
by the composite
\begin{align*}
M^{\alpha_0}e^{-\alpha_0}\overset{\bold D_{\alpha_0,\dots,\alpha_i}}\to 
\bold B_{\alpha_0,\dots, \alpha_i}\otimes M^{\alpha_i}e^{-\alpha_i}
&\overset{1_A^{\otimes i}\otimes d_{M^{\alpha_i}}}
\to \bold B_{\alpha_0,\dots, \alpha_i}\otimes M^{\alpha_i}e^{-\alpha_i} \\
&\overset{1_A^{\otimes i}\otimes \bold D_{\alpha_i,\dots ,\alpha_n}}\to
\bold B_{\alpha_0,\dots, \alpha_n}\otimes M^{\alpha_n}e^{-\alpha_n}.
\end{align*}
Then by the relation (\ref{another expression for DG condition}), 
we have
\begin{align*}
& \sum_{\alpha_i<\beta<\alpha_{i+1}}
(1_A^{\otimes i}\otimes \mu \otimes 1_A^{\otimes (n-i)}\otimes 1_M)
\bold D_{\alpha_0, \dots ,\alpha_i,\beta, \alpha_{i+1},\dots, \alpha_n} \\
= &
-\bold E(\alpha_0,\dots ;\alpha_i;\alpha_{i+1},\dots, \alpha_n)
-\bold E(\alpha_0,\dots ;\alpha_{i+1};\alpha_{i+2},\dots, \alpha_n) \\
&+(-1)^{n-i+1}
(1_A^{\otimes (i-1)}\otimes d_A \otimes 1_A^{\otimes (n-i-1)}\otimes 1_M)
\bold D_{\alpha_0, \dots ,\alpha_i,\alpha_{i+1},\dots, \alpha_n}
\end{align*}
and therefore
the term (\ref{2-2}) is equal to
\begin{align}
\nonumber
&
 \sum_{i=1}^{n-1}\sum_{\alpha_0<\dots<\alpha_i<\alpha_{i+1}<\alpha_{i+2}<\dots <\alpha_n}
(-1)^{i} 
(1_A^{\otimes (n-1)}\otimes t\otimes 1_M)\\
\nonumber
&
(1_A^{\otimes (i-1)}\otimes \mu \otimes 1_A^{\otimes (n-i-1)}\otimes 1_M)
D_{\alpha_0, \dots ,\alpha_i,\alpha_{i+1}, 
\alpha_{i+2},\dots, \alpha_n} =\\
\label{3-1}
&
\sum_{\alpha_0<\dots <\alpha_{n-1}}
D_{\alpha_0,\dots, \alpha_{n-1}}d_M
\\
\label{3-2}
&
-\sum_{\alpha_0<\dots <\alpha_{n-1}}
(1^{\otimes (n-1)}\otimes d_M)D_{\alpha_0,\dots, \alpha_{n-1}} \\
\label{3-3}
&- 
 \sum_{i=1}^{n-1}\sum_{\alpha_0<\dots <\alpha_{n-1}}
(1_A^{\otimes (i-1)}\otimes d_A \otimes 1_A^{\otimes (n-i-1)}\otimes 1_M)
D_{\alpha_0,\dots, \alpha_{n-1}}.
\end{align}
Since
\begin{align*}
&\text{(\ref{1-3})}+\text{(\ref{2-3})}+\text{(\ref{3-2})}=0, \quad
\text{(\ref{2-1})}+\text{(\ref{3-1})}=\Delta \delta_M,  \\
& \text{(\ref{1-2})}+\text{(\ref{3-3})}=0,
\end{align*}
we have
$\Delta\delta_M=d\Delta.$
\end{proof}

\subsubsection{Definition of $\psi:ob(B-com)^{red,b}\to 
ob(K^b\Cal C_A)$} 
Let $N$ be a $B$-comodule.
By the homomorphism $\Delta:N \to B\otimes N$, we have degree preserving
linear maps
\begin{align*}
&\Delta^{(0)}:N \to \oplus_{\alpha_0}\bold k\overset{\alpha_0}\otimes N, \\
&\Delta^{(1)}:N \to \oplus_{\alpha_0<\alpha_1}
\bold k\overset{\alpha_0}\otimes  A\overset{\alpha_1}\otimes  Ne,\\
&\Delta^{(2)}:N \to \oplus_{\alpha_0<\alpha_1<\alpha_2}
\bold k\overset{\alpha_0}\otimes  A\overset{\alpha_1}\otimes  A
\overset{\alpha_2}\otimes Ne^2.
\end{align*}
The $\alpha_0$-component,
the $(\alpha_0,\alpha_1)$-component and
the $(\alpha_0,\alpha_1,\alpha_2)$-component of
$\Delta^{(0)}$, $\Delta^{(1)}$ and $\Delta^{(2)}$ 
are denoted as $p_{\alpha_0}$,
$D_{\alpha_0,\alpha_1}$ and $D_{\alpha_0,\alpha_1,\alpha_2}$, respectively.
\begin{lemma}
\begin{enumerate}
\item
The linear maps $p_{\alpha_0}$ are complete projection orthogonal to 
each other. Moreover these orthogonal projections defines a
direct sum decomposition of $N$.
\item
We have the following equalities
\begin{align}
\nonumber
(1\otimes p_{\alpha_1})D_{\alpha_0,\alpha_1} & =
D_{\alpha_0,\alpha_1}=
D_{\alpha_0,\alpha_1}p_{\alpha_0},
\\
\label{composite of connection}
D_{\alpha_0,\alpha_1,\alpha_2}& =(1\otimes D_{\alpha_1,\alpha_2})
D_{\alpha_0,\alpha_1}.
\end{align}
\end{enumerate}
\end{lemma}
\begin{proof}
Since the composite map $N \overset{\Delta}\to B\otimes N
\overset{\epsilon\otimes 1}\to N$ is the identity, we have 
$\sum_{\alpha_0}p_{\alpha_0}=id_N$.
For $x \in N$, we have
\begin{align*}
(\Delta_B\otimes 1)(\Delta_N^{(0)}(x))&=
\sum_{\alpha_0}(\Delta_B\otimes 1)
((1\overset{\alpha_0}\otimes 1)\otimes p_{\alpha}(x)) \\
&=\sum_{\alpha_0}(1\overset{\alpha_0}\otimes 1)\otimes
(1\overset{\alpha_0}\otimes 1)\otimes p_{\alpha_0}(x)
\end{align*}
and
\begin{align*}
(1\otimes \Delta_N^{(0)})(\Delta_N^{(0)}(x)) &=
\sum_{\alpha_0}(1\otimes \Delta_N^{(0)})
((1\overset{\alpha_0}\otimes 1)\otimes p_{\alpha_0}(x)) \\
& =\sum_{\alpha_0<\alpha_1}
(1\overset{\alpha_0}\otimes 1)\otimes
(1\overset{\alpha_1}\otimes 1)\otimes p_{\alpha_1}p_{\alpha_0}(x).
\end{align*}
Thus by the coassociativity, $p_{\alpha_0}$ is a complete system of
orthogonal projection.

We show the second statement. For $x \in N$, we have
\begin{align*}
(\Delta_B\otimes 1)(D_{\alpha_0,\alpha_1}(x))=&
(1\overset{\alpha_0}\otimes 1)\otimes D_{\alpha_0,\alpha_1}(x)
+D_{\alpha_0,\alpha_1}(x)\otimes (1\overset{\alpha_1}\otimes 1) \\
& \in (\bold k\overset{\alpha_0}\otimes \bold k) \otimes 
(\bold k\overset{\alpha_0}\otimes A\overset{\alpha_1}\otimes \bold k)
\otimes  N \\
& \oplus
(\bold k\overset{\alpha_0}\otimes A\overset{\alpha_1}\otimes \bold k)
\otimes (\bold k\overset{\alpha_1}\otimes \bold k) \otimes N
\end{align*}
and
\begin{align*}
(1\otimes\Delta^{(0)})(D_{\alpha_0,\alpha_1}(x))= &
\sum_{\gamma}
(1\otimes p_{\gamma})(D_{\alpha_0,\alpha_1}(x))\otimes 
(1\overset{\gamma}\otimes 1) \\
\in (\bold k\overset{\alpha_0}\otimes A\overset{\alpha_1}\otimes N)
\otimes (\bold k\overset{\gamma}\otimes \bold k) 
&=
(\bold k\overset{\alpha_0}\otimes A\overset{\alpha_1}\otimes \bold k)
\otimes (\bold k\overset{\gamma}\otimes \bold k) \otimes N
\end{align*}
\begin{align*}
(1\otimes\Delta^{(1)})((1\overset{\alpha_0}\otimes 1)
\otimes p_{\alpha_0}(x))= &
\sum_{\beta,\gamma}(1\overset{\alpha_0}\otimes 1)
\otimes D_{\beta,\gamma}(p_{\alpha_0}(x)) \\
&\in
(\bold k\overset{\alpha_0}\otimes \bold k)\otimes  
(\bold k\overset{\beta}\otimes A\overset{\gamma}\otimes \bold k)
\otimes N.
\end{align*}
By comparing the direct sum component, we have the lemma.

As for the equality (\ref{composite of connection}),
we consider the 
$$
(\bold k\overset{\alpha_0}\otimes A
\overset{\alpha_1}\otimes \bold k)\otimes 
(\bold k\overset{\alpha_1}\otimes A
\overset{\alpha_2}\otimes \bold k)\otimes N
$$
part of $(\Delta_B\otimes 1)\Delta=(1\otimes \Delta)\Delta$.
The component of $(\Delta_B\otimes 1)\Delta(x)$ is equal to
$D_{\alpha_0,\alpha_1,\alpha_2}(x)$.
On the other hand, this component of $(1\otimes \Delta)\Delta$
is equal to $(1\otimes D_{\alpha_1,\alpha_2})D_{\alpha_0,\alpha_1}$.
\end{proof}

Now we construct the correspondence $\psi$.
Let $M^{i}$ be the direct sum component of the complex $Ne^i$ 
associated to the projector $p_i$. We assume that $N$ is
a bounded $B$ comodule.
We make an object $(\{M^i\},\{d_{ji}\})$ of $K^b\Cal C_A$ from
the sequence of complex $M^i$.
By (2) of the above lemma, $D_{\alpha_0,\alpha_1}$
is regarded as a degree zero linear map 
$M^{\alpha_0}e^{-\alpha_0}\to 
(A\otimes M^{\alpha_1})e^{-\alpha_1+1}$.
Therefore $D_{\alpha_0,\alpha_1}$ defines an element
of $d_{\alpha_1,\alpha_0}\in\underline{Hom}^{1}_{\Cal C_A}
(M^{\alpha_0}e^{-\alpha_0},M^{\alpha_1}e^{-\alpha_1})$.
\begin{proposition}
The pair $(M^{\alpha_0},d_{\alpha_1,\alpha_0})$
defined as above is a bounded DG complex in $\Cal C_A$.
\end{proposition}
\begin{proof}
It is enough to show the relation
(\ref{another expression for DG condition}).
It is obtained by considering 
the
$
(\bold k\overset{\alpha_0}\otimes A 
\overset{\alpha_2}\otimes \bold k)
\otimes M^{\alpha_2}
$-component of the
equality $\Delta(dx)=d(\Delta(x))$ restricted
to $M^{\alpha_0}$.
\end{proof}

\subsection{The functor $\psi:(B-com)^{red,b}\to K^b\Cal C_{A}$ 
on Morphisms}
\subsubsection{$B$-comodules and morphisms in $K\Cal C_A$}

In this section, we construct the functor
$\psi:(B-com)^{red,b}\to K^b\Cal C_{A}$
for morphisms in Theorem \ref{first main theorem}.
In this subsection, we set $B_{red}=B_{red}(A,\epsilon)$.
\begin{definition}
\begin{enumerate}
\item
Let $p$ be the projector
$B_{red}\to Ie$ and $p_{n,\beta}$ be the homomorphism 
$$p_{n,\beta}:
B_{red}^{\otimes n}\otimes Ne^{-n}
\to I \otimes\cdots 
\otimes I \otimes N^{\beta}e^{-\beta}
$$
defined by
$p_{n,\beta}=p\otimes \cdots \otimes
p\otimes p_{\beta}$.
\item
Let $M,N$ be objects in $(B-com)^{red,b}$,
$\psi(M)=\{M^i\}$, $\psi(N)=\{N^i\}$
the corresponding objects in $K^b\Cal C_{A}$
and $f=\sum_nf^{(n)}$ 
an element in 
$$\underline{Hom}_{B_{red}-com}^i
(M_{red},N_{red})=
\oplus_n
\underline{Hom}_{KVect_{\bold k}}^i(M,
B_{red}^{\otimes n}\otimes Ne^{-n})
$$
For integers $\alpha<\beta$,$n\geq 0$ 
we define a homomorphism
\begin{align*}
\Psi(f)_{\alpha, \beta}^{(n)}\in
\underline{Hom}_{KVec_{\bold k}}^i
(M^\alpha e^{-\alpha},A\otimes N^\beta e^{-\beta}) 
\end{align*}
by the composite
\begin{align*}
M^{\alpha_0}e^{-\alpha_0}
\to M 
&\overset{f^{(n)}}\longrightarrow
B^{\otimes n}_{red}\otimes 
N e^{-n}\\
&
\overset{p_{n,\beta}}
\longrightarrow
I
\otimes\cdots 
\otimes I
\otimes 
N^{\beta}e^{-\beta} \\
&\overset{(\text{product})\otimes 1}\longrightarrow
A
\otimes 
N^{\beta}e^{-\beta}.
\end{align*}
\item
Let $f$ an element in $\underline{Hom}_{B_{red}-com}^i
(M_{red},N_{red})$.
We define an element $\psi(f)$ of 
$\underline{Hom}_{K\Cal C_A}^i(\psi(M),\psi(N))$ by
$$
\psi(f)_{\alpha,\beta}=
\sum_{n\geq 0}
\Psi(f)_{\alpha,\beta}^{(n)}.
$$
\end{enumerate}
\end{definition}
\begin{proposition}
\begin{enumerate}
\item
Let $M,N$ be objects $(B-com)^{red,b}$.
Then the map
$$
\psi:\underline{Hom}_{B_{red}-com}^{\bullet}(M_{red},N_{red})\to 
\underline{Hom}_{K\Cal C_{A}}^{\bullet}(\psi(M),\psi(N))
$$ 
is a homomorphism of complex and quasi-isomorphism.
\item
The above map is compatible with
composite, that is, the following diagram commutes:

$$
\begin{matrix}
\underline{Hom}_{B_{red}-com}(M_{red},N_{red})
\otimes\underline{Hom}_{B_{red}-com}(L_{red},M_{red})&\to&
\underline{Hom}_{B_{red}-com}(L_{red},N_{red}) \\
\psi\otimes \psi \downarrow & & \downarrow \psi \\
\underline{Hom}_{K\Cal C_A}(\psi(M),\psi(N))
\otimes\underline{Hom}_{K\Cal C_A}(\psi(L),\psi(M))&\to&
\underline{Hom}_{K\Cal C_A}(\psi(L),\psi(N))  
\end{matrix}
$$
\end{enumerate}
\end{proposition}
\begin{proof}
Let $M,N$ be objects in $(B-com)^{red,b}$ and set 
$\psi(M)=\{M_i,\partial_{M,ij}\}$ and
$\psi(N)=\{N_i,\partial_{M,ij}\}$.
Then we have $\partial_{M,ij}=d_{M,ij}+\delta_{M,ij}$,
where 
\begin{align*}
&d_{M,ij}=p_j\circ d\circ \iota_i:M^i\to M \overset{d}\to M \to M^j \\
&\delta_{M,ij}=(p\otimes p_j)\circ\Delta_{M_{red}}\circ \iota_i:
M^i\to M \overset{\Delta_{M_{red}}}\to B_{red} \otimes M \to I\otimes M^j.
\end{align*}
Let $d_H$ be the outer differential (\ref{outer hochschild})
for $\underline{Hom}_{B_{red}-com}(M,N)$ and
$$
d_{B^{\otimes n}}(f)=\sum_{i=1}^n
(1_B^{\otimes(i-1)}\otimes d_{B,red}\otimes 1_B^{\otimes(n-i)}
\otimes N)\circ f\in 
\underline{Hom}(M,B^{\otimes n}_{red}\otimes N),
$$
where $d_{B,red}$ is defined in (\ref{reduced outer}).
To prove the compatibility of differential for the map $\psi$,
it is enough to show the following equality.
\begin{align*}
&\psi(d_H(f)+d_{B^{\otimes n}}(f))_{pq} \\
=&\sum_{p<k<q}
(\mu\otimes 1_{N^q})
\{(1_A\otimes \delta_{N,kq})\psi(f)_{pk}
-(-1)^i(1_A\otimes \psi(f)_{kq})\delta_{M,pk}\}.
\end{align*}
We can check this equality by the definition of $\psi$.
As for the compatibility of composite, it is enough to check the equality:
$$
\psi(f\circ g)_{pq}=\sum_{p<k<q}(\mu\otimes 1_{N^q})(1_A\otimes \psi(f)_{kq})
\psi(g)_{pk}.
$$
This equality also follows from the definition of $\psi$.

To show that the homomorphism $\psi$ is a quasi-isomorphism,
we introduce a finite filtration $Fil^{\bullet}$ on $N$
by $Fil^{\alpha}N=\oplus_{i\geq \alpha}N^ie^{-i}$.
Since the quotient $Fil^{\alpha}N/Fil^{\alpha+1}N$ has a trivial $B$-coaction,
we may assume that the action on $N$ is trivial
by considering the long exact sequence for the short exact sequence
$$
0\to 
Fil^{\alpha+1}N\to 
Fil^{\alpha}N\to Fil^{\alpha}N/Fil^{\alpha+1}N
\to 0.
$$
By similar argument, we may assume
 that
$M$ has also trivial $B$-coaction. Thus the proposition
is reduce to the case where $M=M^{\alpha}e^{-\alpha}=N=
N^{\beta}e^{-\beta}=\bold Q$.
In this case,
we can check by direct calculation.
\end{proof}

\section{$H^0(B(\epsilon\mid A \mid \epsilon))$-comodule
for Connected DGA}

A DGA $A^{\bullet}$ (resp. DG coalgebra $B^{\bullet}$) 
is said to be connected if 
$H^{i}(A^{\bullet})=0$ for $i<0$
and $H^{0}(A^{\bullet})=0$ (resp.
$H^{i}(B^{\bullet})=0$ for $i<0$
and $H^{0}(B^{\bullet})=0$).
Let $A^{\bullet}$ be a connected DGA and 
$\epsilon:A^{\bullet} \to \bold k$ be an augmentation of $A^{\bullet}$.
Then by the bar spectral sequence, 
the DG coalgebra $B(\epsilon \mid A \mid \epsilon)$ is connected.

\begin{lemma}
\label{strictly connectedness}
Let $A$ be a connected DGA. Then there exists a subalgebra $\overline{A}$
of $A$ such that the natural inclusion $\overline{A}\to A$ is a
quasi-isomorphism and
\begin{align*}
&\text{(p)}\quad A^{-i}=0 \text{ for }i<0,
\text{ (positive condition) and } \\
&\text{(r)}\quad \overline{A}^0=\bold k
\text{ (rigid condition)}.
\end{align*}
\end{lemma}
\begin{proof}
Let $B^1$ and $Z^1$ be the image and the kernel of $d$ in $A^1$.
Then we have $B^1 \subset Z^1 \subset A^1$. We choose a splitting
$Z^1=B^1\oplus L$. We set $\overline{A}^0=\bold k, \overline{A}^1=L$
and $\overline{A}^i=A^i$ for $i>1$. Then $\overline{A}$ is a sub DGA
of $A$ which is quasi-isomorphic to $A$ and satisfies the conditions
of the lemma.
\end{proof}
\begin{definition}
\begin{enumerate}
\item
We define $K^0\Cal C_A$ (resp. $K^{\leq 0}\Cal C_A$,$K^{\geq 0}\Cal C_A$) 
as the full subcategory of $K^b\Cal C_A$
of objects $M=(M^i,d_{i,j})$ where $M^i$ is of the form
$M^{-i,i}e^i$ 
(resp. $\oplus_{i+j\leq 0}M^{j,i}e^{-j}$,
$\oplus_{i+j\geq 0}M^{j,i}e^{-j}$)
for $\bold k$-vector spaces $M^{-i,i}$.
\item
We define an additive category $H^0K^0\Cal C_A$ as follows.
The class of objects of $H^0K^0\Cal C_A$ is 
the same as in $K^0\Cal C_A$.
For $a,b\in K^0\Cal C_A$, we define the morphism
$$
Hom_{H^0K^0\Cal C_A}(a,b)=H^0(\underline{Hom}_{K^b\Cal C_A}(a,b)).
$$
\end{enumerate}
\end{definition}
\begin{definition}
Let $A^{\bullet}$ be a differential graded algebra.
\begin{enumerate}
\item
A pair $(\Cal M,\nabla)$ of free $A^0$-module $\Cal M$
and a $\bold k$-linear map $\nabla:\Cal M \to \Cal M\otimes_{A^0}A^1$
is called an $\Cal A$ connection if $\nabla(am)=a\nabla(m)+da\cdot m$.
An $\Cal A$ connection is said to be integrable if $\nabla\circ\nabla=0$.
\item
An $A^{\bullet}$ connection $(M,\nabla)$ is said to be trivial
if it is generated by horizontal sections.
\item
An $A^{\bullet}$ connection $(\Cal M,\nabla)$
is called nilpotent if there exists a 
finite filtration by connections $F^p\Cal M$ such that
$Gr_F^p(\Cal M)$ is a trivial connection.
The category of integrable nilpotent connections is denoted
as $(INC_A)$. Morphism is a $A^0$ homomorphism compatible with
connections.
\item
A homomorphism $F\in Hom_{INC_A}(\Cal M,\Cal N)$ is homotopy equivalent if 
there is a map $h:\Cal M\to \Cal N\otimes_{A^0} A^{-1}$ of 
$A^0$-homomorphism
such that $F=\nabla\circ h+h\circ\nabla$. (See the following diagram.)
$$
\begin{matrix}
\Cal M & \overset{h}\to & A^{-1}\otimes_{A^0}\Cal N \\
\nabla_M\downarrow & & \downarrow \nabla_N \\
A^1\otimes_{A^0}\Cal M & \overset{h}\to & A^0\otimes_{A^0}\Cal N \\
\end{matrix}
$$
By localizing the homomorphisms by homotopy equivalent, we have
a category $(HINC_A)$ of homotopy integrable nilpotent connections.
By the definition, if the condition (p) 
is satisfied, then $(INC_{A})$ and $(HINC_{A})$
are equivalent.
\end{enumerate}
\end{definition}
\begin{proposition}
\label{DGcat to connection}
Let $A^{\bullet}$ be a connected DGA
with an augmentation $\epsilon$.
The category $(HINC_A)$ of homotopy equivalence class of 
integrable nilpotent $A^{\bullet}$ connections and
$H^0K^0\Cal C_A$ are equivalent.
\end{proposition}
\begin{proof}
We define a functor $F:K^0\Cal C_A\to (INC_A)$ by 
taking
$(\oplus_i M^{-i,i}\otimes A^0,\nabla)$
for an object $(M^i,d_{i,j})\in K^0\Cal C_A$ with 
$M^i=\oplus_iM^{-i,i}e^i$.
Here $\nabla$ is defined by $(\sum_{i,j}d_{i,j})
\otimes 1+1\otimes d$.
By restricting the above functor, we have a functor
$
Z^0K^0\Cal C_A \to (INC_A).
$
By the definition of $(INC_A)$, 
the following natural homomorphisms are isomorphisms:
\begin{align*}
&Z^0Hom_{K^0\Cal C_A}(M,N)\to Hom_{INC_A}(F(M),F(N)) \\
&H^0Hom_{K^0\Cal C_A}(M,N)\to Hom_{HINC_A}(F(M),F(N))
\end{align*}
Thus the functor $F:Z^0K^0\Cal C_A \to (INC_A)$
and 
$\overline{F}:H^0K^0\Cal C_A \to (HINC_A)$
are fully faithful functors.
We show the essentially surjectivity.
Let $\nabla:M\to A^1\otimes M$ be 
a nilpotent integrable $A$-connection.
Let $F^{\bullet}M$ be a nilpotent filtration of $M$
for the connection $\nabla$. We choose a 
splitting $M\simeq \oplus Gr_F^i(M)$, where
$Gr^i_F(M)=F^i(M)/F^{i+1}(M)$. Let 
$M^{-i,i}=Gr^i_F(M)$ and
$\nabla_{ij}:M^{-i,i}\to A^1\otimes M^{-j,j}$
be the corresponding component of $\nabla$.
We set $M^i=M^{-i,i}e^i$ and
$d_{ij}=\nabla_{ij}\in \underline{Hom}^1_{\Cal C_A}
(M^{i}e^{-i},M^{j}e^{-j})$. Then $(\{M^i\},d_{ij})$
is an object of $K^0\Cal C_A$.
\end{proof}
We define the following homomorphism (1) 
by taking the fiber of $F(M)$ at
the augmentation $\epsilon$:
\begin{align*}
Z^0Hom_{K^0\Cal C_A}(M,N)
\overset{(1)}\longrightarrow 
&Hom_{\bold k}
(F(M)\otimes_{\epsilon} \bold k,F(N) \otimes_{\epsilon} \bold k) \\
& =
Hom_{\bold k}
(\oplus M^{-i,i},\oplus N^{-j,j}).
\end{align*}
\begin{lemma}
The image of (1) is identified with 
$H^0Hom_{K^0\Cal C_A}(M,N)$.
\end{lemma}
\begin{proof}
Since the image of $B^0Hom_{K^0\Cal C}(M,N)$ under the functor 
(1) is zero, the functor (1) factors through
$$
H^0Hom_{K^0\Cal C}(M,N)\overset{(2)}\to Hom_{\bold k}
(\oplus_i M^{-i,i},\oplus_j N^{-j,j}).
$$
We show the injectivity of (2) by the induction of 
the nilpotent length of the corresponding connections.

Let $M$ and $N$ be $k$-vector spaces.
Since $A$ is connected, the natural homomorphism
$$
H^0Hom_{K^0\Cal C_A}(M, N)\to
Hom_{K^0\Cal C_A}(\oplus_iM^{-i,i},\oplus_jN^{-j,j} )
$$
is an isomorphism.
Let $N^{\geq i}$ and $N^{\leq i}$ be the stupid filtration of $N$ as
DG complex in $\Cal C_A$ and the quotient of $N$ by $N^{\geq i}$.
Since $A$ is connected, we have the following left exact sequences
\begin{align*}
0\to H^0Hom_{K^0\Cal C_A}(M, N^{\geq i}) &\to
H^0Hom_{K^0\Cal C_A}(M, N) \\
&\to
H^0Hom_{K^0\Cal C_A}(M, N^{\leq i}), \\
0\to H^0Hom_{K^0\Cal C_A}(M^{\leq i}, N) &\to
H^0Hom_{K^0\Cal C_A}(M, N) \\
&\to
H^0Hom_{K^0\Cal C_A}(M^{\geq i}, N).
\end{align*}
By induction on the nilpotent length of $N$ and $M$, 
we have the lemma by 5-lemma.
\end{proof}

\begin{theorem}
\label{connction theorem}
Let $A$ be a connected DGA.
The category $H^0K^0\Cal C_A$ is equivalent to the category
of nilpotent $H^0(B(\epsilon\mid A\mid \epsilon))$-comodules.
\end{theorem}
\begin{proof}
By Lemma \ref{strictly connectedness}, we choose a quasi-isomorphic
sub DGA $A'$
of $A^{\bullet}$ such that
${A'}^{-i}=0$ for $i<0$ and ${A'}^0\simeq \bold k$. 
In this case, the condition (p) is satisfied and the categories
$(INC_{A'})$ and $(HINC_{A'})$ are equivalent. 
Thus we have the following commutative diagram of categories:
$$
\begin{matrix}
(INC_{A'}) & \overset{\alpha}\to & (INC_A) \\
\downarrow & & \downarrow \\
(HINC_{A'}) & \overset{\alpha'}\to & (HINC_A) \\
F_{A'}\downarrow & & \downarrow F_A\\
(H^0(B_{A'})-\text{comod}) &\overset{\alpha''}\to & 
(H^0(B_{A})-\text{comod}).
\end{matrix}
$$
We know that $\alpha'$ is fully faithful and $\alpha''$ is equivalent.
Thus it is enough to show that $F_A$ is fully faithful and 
$F_{A'}$ is equivalent.

We show the fully faithfulness of $F_A$. 
Let $\Cal N,\Cal M$ be $B$-comodule corresponding to the
nilpotent integrable $A$ connections $M$ and $N$.
We set $H^0=H^0(B(\epsilon\mid A\mid \epsilon))$.
By the induction of the nilpotent length, we can show that 
the natural map
$$
H^i\underline{Hom}_{B-com}(\Cal M,\Cal N) 
\to H^i\underline{Hom}_{H^0-com}(H^0(\Cal M),H^0(\Cal N))
$$
is an isomorphism for $i=0$ and injective for $i=1$
using 5-lemma.
Therefore the functor $F_{A}$ is a fully faithful functor.

We introduce a universal
integrable nilpotent connection on $H^0(B(\epsilon\mid A'\mid \epsilon))$
to show the essential surjectivity of $(HINC_{A'})\to (H^0(B_{A'})-comod)$.
We use the isomorphism 
$H^0(B_{red}(A', \epsilon))\simeq
H^0(B(\epsilon\mid A'\mid \epsilon))$
between simplicial bar complex
and Chen's reduced bar complex. 
By the conditions (p) and (r),
$H^0(B_{red}(A', \epsilon))$
 is identified with a subspace
of $B_{red}(A', \epsilon)^0$.
The coproduct on $B_{red}(A',\epsilon)$ induces a
connection
$$
H^0(B_{red}(A',\epsilon))\to (A')^1\otimes H^0(B_{red}(A',\epsilon))
$$
on $H^0=H^0(B_{red}(A',\epsilon))$.
For an $H^0$-comodule $(M,\Delta_M)$,
We define $\Cal M$ as the kernel of the following map:
$$
H^0\otimes M 
\overset{\Delta_{H^0}\otimes 1_M - 1_{H^0}\otimes \Delta_M}\to 
H^0\otimes H^0\otimes M.
$$
Then $\Cal M$ has a structure of $A'$-connection and
$F_{A'}(\Cal M)=M$.
This proves the essentially surjectivity.
\end{proof}
Since the category of nilpotent $H^0$-comodules is stable under 
taking kernels and cokernels, we have the following corollary.
\begin{corollary}
\begin{enumerate}
\item
If $A^{\bullet}$ is a connected DGA, then $\Cal A=H^0K^0\Cal C_A$ is an abelian
category.
\item
Let $A$ and $A'$ be connected
DGA's and 
$\varphi:A\to A'$ a quasi-isomorphism of DGA.
Then the associated categories $(HINC_A)\to (HINC_{A'})$ are
equivalent.
\end{enumerate}
\end{corollary}

\section{Patching of DG category}

In this section, we consider patching of DG categories
and their bar complexes. A typical example for patching appears
as van Kampen's theorem. Let $X$ be a manifold, which is covered by
two open sets $X_1$ and $X_2$. Suppose that $X_{12}=X_1 \cap X_2$ is connected.
Then the fundamental group $\pi_1(X)$ is isomorphic to
the amalgam product 
$\pi_1(X_1)\underset{\pi_1(X_{12})}*
\pi_1(X_2)$. We can interpret this isomorphism as an equivalence of
two categories. The first category $Loc(X)$ is 
a category of local systems
on $X$ and the second is a category
$Loc(X_1)\times_{Loc(X_{12})} Loc(X_2)$
of triples $(\Cal L_1, \Cal L_2,\varphi)$, where
$\Cal L_1,\Cal L_2$ are local systems on $X_1$ and $X_2$ and
$\varphi$ is an isomorphism of local systems 
$\varphi:\Cal L_1\mid_{X_{12}}\to \Cal L_2\mid_{X_{12}}$.

We must be careful in the nilpotent version. Let $Loc^{nil}(X)$
be the category of nilpotent local systems. Then the natural functor
$$
Loc(X)^{nil}\to Loc(X_1)^{nil}\times_{Loc(X_{12})^{nil}}Loc(X_{2})^{nil}
$$
is not an equivalent in general. 
For example, if $X_{12}$ is contractible,
then there might be no filtration on 
$\Cal L_1\mid_{X_{12}}=\Cal L_2\mid_{X_{12}}$ which induces
a nilpotent filtration $F_1$ on
the local system $\Cal L_1$ on $X_1$ and that on the local system
$\Cal L_2$ on $X_2$.

\begin{definition}
\label{def of fiber product of DGAs}
\begin{enumerate}
\item
Let $\Cal A$ be an abelian category. $\Cal A$ is said to be nilpotent
\begin{enumerate}
\item
if there is an object $\bold 1$, and
\item
for any object $M$ in $\Cal A$, there is a filtration 
$F^{\bullet}$ on
$M$ such that the associated graded quotient $Gr^i_F(M)$ is
a direct sum of $\bold 1_A$.
\end{enumerate}
This filtration is called a nilpotent filtration of $M$.
\item
Let $\Cal A_1$ and $\Cal A_2$ be nilpotent abelian categories and 
$\varphi:\Cal A_1 \to \Cal A_2$ be an exact additive functor.
The functor
$\varphi$ is said to be nilpotent 
if for any nilpotent filtration $F^i$ 
on an object $M$, $\varphi(F^i)$ is a nilpotent filtration on
$\varphi(M)$.
\item
Let $\Cal A_2,\Cal A_2,\Cal A_{12}$ be a nilpotent abelian category
and $F_1:\Cal A_1\to \Cal A_{12}$ and $F_2:\Cal A_2\to \Cal A_{12}$
be nilpotent functors.
We define nilpotent fiber product 
$(\Cal A_1\times_{\Cal A_{12}} \Cal A_2)^{nil}$ as the 
full sub category
of $(\Cal A_1\times_{\Cal A_{12}} \Cal A_2)$
consisting of objects $(\Cal L_1, \Cal L_2,\varphi)$
such that there exist nilpotent filtrations $N_1$ and $N_2$
on $\Cal L_1$ and $\Cal L_2$ such that 
$\varphi(F_1(N_1^i))=F_2(N_2^i)$.
\item
Let $A_1,A_2,A_{12}$ be DGA's and $u_1:A_1\to A_{12}$ 
and $u_2:A_{2}\to A_{12}$
be homomorphisms of DGA's. We define a fiber product 
$\widetilde{A}=A_1\times_{A_{12}} A_2$ as
$\widetilde{A}^p=A_1^p \oplus A_2^p \oplus A_{12}^{p-1}$.
We introduce a product of elements
$a=(a_1,a_2,a_{12})$ and $b=(b_1,b_2,b_{12})$ of
$\widetilde{A}^p$ and $\widetilde{A}^q$ as
$$
a\cdot b=(a_1\cdot b_1,a_2\cdot b_2, u_1(a_{1})\cdot b_{12}
+(-1)^{\deg b_2}a_{12}\cdot u_2(b_2)).
$$
By setting 
$s(A_1\times_{A_{12}}A_2)=A_1\oplus A_2\oplus A_{12}e^{-1}$,
this rule can be written as the following simpler rule
$$
(a_1+a_2+a_{12}e^{-1})(b_1+b_2+b_{12}e^{-1})=
a_1b_1+a_2b_2+u_1(a_1)b_{12}e^{-1}+a_{12}e^{-1}u_2(b_2).
$$
\end{enumerate}
\end{definition}
\begin{proposition}
Let $A_1,A_2$ and $A_{12}$ be connected DGA's,
$u_1:A_1\to A_{12}$ and $u_2:A_2\to A_{12}$ homomorphisms of
DGA's and $\widetilde{A}=A_{1}\times_{A_{12}}A_2$.
Assume that there exists 
an augmentation $\epsilon:A_{12}\to \bold k$.
Then the natural functor
$$
H^0K^0\Cal C_{\tilde A} \to (H^0K^0\Cal C_{A_1}
\times_{H^0K^0\Cal C_{A_{12}}}H^0K^0\Cal C_{A_2})^{nil}
$$
is an equivalence of categories.
\end{proposition}
\begin{proof}
The fully faithfulness follows from the direct calculation.
We show the essentially surjectivity.
Since $A_\star$ is connected, $H^0K^0\Cal C_{A_\star}$
is equivalent to the category $(HINC_{A_{\star}})$ for $\star=1,2,12$.
Let $M=(M_1, M_2,\varphi)$ be an object of \linebreak
$(HINC_{A_1}\times_{HINC_{A_{12}}}HINC_{A_2})^{nil}$,
where
$$M_p=(M_p,\nabla_p:M_p\to A^1_p\otimes_{A_p^0}M_p)
\in HINC_{A_p}\quad (p=1,2)
$$
is a nilpotent integrable connection.
By the definition of nilpotent fiber product,
there exist filtrations $F_1^{\bullet}$ and 
$F_2^{\bullet}$ on $M_1$
and $M_2$ and an $A^0_{12}$-isomorphism 
$\varphi:u_1(M_1) \to u_2(M_2)$ such that
(1) $\varphi$ is compatible with the connections on $u_1(M_1)$
and $u_2(M_2)$, and (2) the filtrations $u_1(F_1^{\bullet})$
and $u_2(F_2^{\bullet})$ are isomorphic via the isomorphism
$\varphi$.
By the isomorphism, we have an identification
$\bold k\otimes_{\epsilon,A_1^0}M_1\simeq
\bold k\otimes_{\epsilon,A_2^0}M_2=M^{(0)}$.
The $\bold k$-vector space $M^{(0)}$ has a filtration $F^{\bullet}$
induced by
$u_1(F^{\bullet}_1)=u_2(F_2^{\bullet})$. 
Using identification 
$A^0_p\otimes_{\bold k}M^{(0)} \simeq M_p$ 
we have a map $\overline{\nabla}_p:M^{(0)}\to 
A^1_p\otimes_{\bold k}M^{(0)}$
for $p=1,2$. On the other hand, the morphism $\varphi$ defines a
map $\overline{\varphi}:M^{(0)}\subset u_1(M_1)\to 
u_2(M_2)\simeq M^{(0)}\otimes A^0_{12}$.
Thus we have a map 
$$
M^{(0)}\overset{(\overline{\nabla}_1,
\overline{\nabla}_2,\overline{\varphi})}
\longrightarrow (A^1_1\oplus A^1_2 \oplus A^{0}_{12})\otimes M^{(0)}.
$$
By simple calculation, this map gives rise to 
an integrable nilpotent filtration.
This nilpotent filtration give rise to
an object $(M_1,M_2,\varphi)$ of \linebreak
$(HINC_{A_1}\times_{HINC_{A_{12}}}HINC_{A_2})^{nil}$.
\end{proof}

Let $\epsilon:A_{12}\to \bold k$ be an augmentation of $A_{12}$.
By composing the morphism $\varphi_i:A_i \to A_{12}$,
we have a augmentation $\epsilon_i=\epsilon\circ\varphi_1:A_i \to \bold
k$ for $i=1,2$
By composing the natural projection $\widetilde A \to A_i$, 
we have two augmentations 
$e_i:\widetilde A \to \bold k$
for $i=1,2$.
We introduce a comparison copath morphism
$$
p(\epsilon):B(e_1\mid \widetilde{A}\mid e_2) \to \bold k
$$
connecting $e_1$ and $e_2$
associated to $\epsilon$.
Since 
$$
B(e_1\mid \widetilde{A}\mid e_2)^0=\oplus_{\mid \alpha\mid=k}B_{\alpha}^{-k}
$$
and 
$$
\oplus_{\mid \alpha\mid=0}B_{\alpha}^{0}=
\oplus_{\alpha_0}\bold k\overset{\alpha_0}\otimes \bold k
$$
\begin{align*}
\oplus_{\mid \alpha\mid=1}B_{\alpha}^{1} & =
\oplus_{\alpha_0<\alpha_1}\bold k\overset{\alpha_0}\otimes 
\widetilde{A}^1\overset{\alpha_1}\otimes \bold k\\
& =
\oplus_{\alpha_0<\alpha_1}\bold k\overset{\alpha_0}\otimes 
(A_1^1\oplus A_2^1\oplus A_{12}^0)\overset{\alpha_1}\otimes \bold k.
\end{align*}
We define a linear map 
$p(\epsilon):B(e_1\mid \widetilde{A}\mid e_2)^0\to \bold k$
by 
the summation of
(1)
projection to the factor
$\bold k\overset{\alpha_0}\otimes \bold k$ over $\alpha_0$, and
(2)
the composite of
$\epsilon$ and the projection to
$\bold k\overset{\alpha_0}\otimes 
A_{12}^0\overset{\alpha_1}\otimes \bold k$
over $\alpha_0<\alpha_1$. 
We can show that the composite 
$$
B(e_1\mid \widetilde{A}\mid e_2)^{-1}\to 
B(e_1\mid \widetilde{A}\mid e_2)^0\to \bold k
$$
is zero. For example $1\overset{\alpha_0}\otimes x 
\overset{\alpha_1}\otimes 1\in
\bold k\overset{\alpha_0}\otimes A_1^0 
\overset{\alpha_1}\otimes \bold k$ goes to
$$
\epsilon_1(x)\overset{\alpha_1}\otimes 1 +
1 \overset{\alpha_0}\otimes dx\overset{\alpha_1}\otimes 1 +
1 \overset{\alpha_0}\otimes\varphi_1(x)\overset{\alpha_1}\otimes 1 
\in B(e_1\mid \widetilde{A}\mid e_2)^0,
$$
which is an element of the kernel of $p(\epsilon)$.
\begin{definition}
The map $p(\epsilon)$ is called the comparison copath associated to
$\epsilon$.
\end{definition}

\section{Graded case}
We consider a graded version of DGA. Let 
$A=\oplus_{k\geq 0}A_k=\oplus_{k\geq 0}A_k^{\bullet}$ be
a graded DGA over a field $\bold k$. We assume that the image of 
$A_p\otimes A_q$ under the 
multiplication map is contained in $A_{p+q}$.
We introduce a DG category $\Cal C_A^{gr}$ as follows.
Objects are finite direct sum of the form $V^{\bullet}(i)$ where
$V^{\bullet}$ is a complex of $\bold k$-vector space. For objects 
$V^{\bullet}(i)$ and $W^{\bullet}(j)$,
we define
\begin{align*}
\underline{Hom}_{\Cal C^{gr}_A}^{\bullet}
(V^{\bullet}(i),W^{\bullet}(j))=
\begin{cases}
\underline{Hom}_{KVec_{\bold k}}^{\bullet}
(V^{\bullet},A_{j-i}\otimes W^{\bullet})
 &\text{ if }j\geq i \\
0 &\text{ if }j<i. 
\end{cases}
\end{align*}
Then $\Cal C^{gr}_A$ becomes a DG category. The DG category of DG
complexes in $\Cal C^{gr}_A$ is denoted as $K\Cal C_A^{gr}$.
The category $\Cal C_{A^{gr}}$ has a tensor structure
by $V(p)\otimes W(q) = (V\otimes W)(p+q)$.
The full subcategory of the bounded DG complexes is denoted as
$K^b\Cal C_A^{gr}$. By considering $\bold k$ as a graded DGA 
in a trivial way, we have a DG-category $\Cal C_{\bold k}^{gr}$.
For example, for $\bold k$-vector spaces $V,W$, we have
$$
\underline{Hom}_{\Cal C_{\bold k}^{gr}}(V(p),W(q))=
\begin{cases}Hom(V, W) & \text{ if }p=q \\
0 & \text{otherwise}.
\end{cases}
$$
The object $V(p)$ is called the $p$-Tate twist of $V$ and the Tate
weight of $V(p)$ is defined to be $p$.
The category $\Cal C_{\bold k}^{gr}$ is the category of
formal Tate twist of $k$-vector spaces.
\begin{definition}
Let $M=\oplus_i M^{\bullet}_i$
be a graded complex of $k$-vector spaces.
We introduce a homogenization $M^h$ of 
$M$ 
by $\oplus_i(M_i\otimes \bold k(-i))$
as an object in $\Cal C_{\bold k}^{gr}$. 
By this correspondence, the category of graded complex
is equivalent to the category
$\Cal C^{gr}_{\bold k}$.
Under this identification, the object 
$\bold k(p)\otimes \bold k(q)$ is identified with
$\bold k(p+q)$.
\end{definition}
Using the above notations, for
$V^{\bullet}(p),W^{\bullet}(q)\in \Cal C^{gr}_A$, we have
$$
\underline{Hom}_{\Cal C_A^{gr}}(V^{\bullet}(p),W^{\bullet}(q))
=\underline{Hom}_{\Cal C_{\bold k}^{gr}}(V^{\bullet}(p),
A^h\otimes W^{\bullet}(q)).
$$
Let $\epsilon_0:A_0\to \bold k$ be an augmentation of 
$A_0$ and
$\epsilon:A \to \bold k$ a composite of $\epsilon_0$
and the natural projection.
We define the homogeneous bar complex 
$\bold B^{h}(\epsilon\mid A\mid \epsilon)$ in $K\Cal C_{\bold k}$
as follows.
Let $\alpha=(\alpha_0<\cdots< \alpha_n)$ be a sequence of integers.
We define $\bold B_{\alpha}^h\in \Cal C_{\bold k}^{gr}$ by
$$
\bold k\overset{\alpha_0}\otimes A^{h}
\overset{\alpha_1}\otimes \cdots \overset{\alpha_{n-1}}\otimes A^{h}
\overset{\alpha_n}\otimes\bold k,
$$
and $\bold B^{h}=\bold B^{h}(\epsilon\mid A \mid\epsilon)
\in K\Cal C_{\bold k}^{gr}$ by 
the complex
$$
\cdots\to\oplus_{\mid\alpha\mid=1}\bold B_{\alpha}^h\to
\oplus_{\mid\alpha\mid=0}\bold B_{\alpha}^h\to 0.
$$
The associate simple object in $\Cal C_{\bold k}^{gr}$
is denoted as $B^{h}=B^h(\epsilon\mid A \mid \epsilon)$.
Then the Tate weight $w$ part of $B^{h}$ is equal to
the sum of
$$
(\bold k\overset{\alpha_0}\otimes A^{\bullet}_{p_1}
\overset{\alpha_1}\otimes \cdots 
\overset{\alpha_{n-1}}\otimes A^{\bullet}_{p_n}
\overset{\alpha_n}\otimes\bold k )\otimes
\bold Q(-p_1-\cdots -p_{n-1})
$$
where $w=-(p_1+\cdots+p_n)$. 
The weight $w$ part of $B^{h}$
is denoted as $B^{h}(w)$.
We define the homogenized reduced bar complex
$B_{red}^h=B^h_{red}(A,\epsilon)$ in the same way.

Thus $B^h$, $B_{red}^h$ are graded DG coalgebras in 
$\Cal C_{\bold k}$.
We introduce DG category structure on bounded homogeneous 
$B^{h}$ comodules.
\begin{definition}
\begin{enumerate}
\item
Let $V^{\bullet}=\oplus V^{\bullet}(i)$ 
be an object in $\Cal C_{\bold k}^{gr}$.
A homomorphism 
$\Delta:V^{\bullet}\to B^h\otimes V^{\bullet}$ in 
$\Cal C^{gr}_{\bold k}$ 
(i.e. homomorphism of graded complex)
is called 
a coproduct if it satisfies the coassociativity and
counitarity defined in Definition \ref{def of B comod as DG cat}. 
An object $V^{\bullet}$ in $\Cal C_{\bold k}^{gr}$
equipped with a coproduct
is called a homogeneous $B^h$-comodule. We define
$B^h_{red}$ comodules in the same way.
\item
Let $V^{\bullet}$ and $W^{\bullet}$ be homogeneous $B^h$ comodules.
A morphism $f$ in 
$\underline{Hom}_{\Cal C_{\bold k}^{gr}}(V^{\bullet}, W^{\bullet})$
is called a $B^h$ homomorphism if the diagram
$$
\begin{matrix}
V^{\bullet} &\overset{f}\to &W^{\bullet} \\
\Delta_V\downarrow & & \downarrow\Delta_W \\
B^h\otimes V^{\bullet} & \overset{1_{B^h}\otimes f}
\to &B^h\otimes W^{\bullet} \\
\end{matrix}
$$
is commutative. 
\end{enumerate}
\end{definition}

Let $V=V^{\bullet},W=W^{\bullet}$ be homogeneous $B^h$ comodules.
We define the complex of homomorphism 
$\underline{Hom}_{B^h-com}(V,W)$
from $V$ to $W$
by the associate simple complex
of
$$
0\to \underline{Hom}_{\Cal C^{gr}_{\bold k}}(V,W)
\overset{d_H}\to
\underline{Hom}_{\Cal C^{gr}_{\bold k}}(V,B^h\otimes W)
\overset{d_H}\to
\underline{Hom}_{\Cal C^{gr}_{\bold k}}
(V,{B^h}^{\otimes 2}\otimes W)
\overset{d_H}\to \cdots
$$
where $d_H$ is defined similarly to
(\ref{outer hochschild}).
We define the DG category $(B^h-com)^b$ as the category of
bounded homogeneous $B^h$-comodules with 
$\underline{Hom}_{B^h-com}(\bullet, \bullet)$
as the complex of morphisms.
The following proposition is a graded version of Theorem 
\ref{first main theorem}, and Theorem \ref{connction theorem}

\begin{proposition}
\label{main theorem for graded case}
\begin{enumerate}
\item
The DG category $K^b\Cal C_{A}^{gr}$ is homotopy equivalent to
the category of bounded homogeneous 
$B^h(\epsilon\mid A\mid \epsilon)$ comodules.
\item
Assume that the graded DGA $A$ is connected.
Then $H^0K^{0}\Cal C_{A}^{gr}$ is 
equivalent to the category of homogeneous 
$H^0(B^h(\epsilon\mid A\mid \epsilon))$
comodules.
As a consequence, $H^0K^0\Cal C_{A}^{gr}$ is
an abelian category.
\end{enumerate}
\end{proposition}
\begin{proof}
(1) We set $B^h=B^h(\epsilon \mid A\mid \epsilon)$.
We give a one to one correspondence $ob(B^h-comod)^b\to 
ob(K^b\Cal C^{gr}_{\bold k})$ of objects.
Let $V=\oplus_iV^{\bullet}(i)$ be an object in
$\Cal C^{gr}_{\bold k}$ and
$V\to B^h\otimes V$ be the comultiplication of $V$.
We use the direct sum decomposition
$$
B^h=(\oplus_{\alpha_0}\bold k \overset{\alpha_0}\otimes \bold k)
\oplus
(\oplus_{\alpha_0<\alpha_1}\bold k\overset{\alpha_0}\otimes A^h
\overset{\alpha_1}\otimes \bold ke)
\oplus
\bigoplus_{\mid\alpha \mid=i>1}\bold B^h_{\alpha}e^i.
$$
Let $\pi_{\alpha_0}$ be the projection $B^h\to 
\bold k\overset{\alpha_0}\otimes\bold k$.
Let $p_{\alpha_0}$ be the composite of the map 
$$
V\overset{\Delta_V}\to B\otimes V\overset{\pi_{\alpha_0}}\to
\bold k \overset{\alpha_0}\otimes V
$$
and $V^{\alpha_0}$ be $Im(p_{\alpha_0})e^{\alpha_0}$.
Since the map $p_{\alpha_0}$ is homogeneous,
$V^{\alpha_0}$ is an object in $\Cal C_{\bold k}^{gr}$.
By the assumption of boundedness, $V^{\alpha_0}=0$
except for finite numbers of $\alpha_0$.
Let
$\pi_{\alpha_0,\alpha_1}$ be the projection $B^h\to 
\bold k\overset{\alpha_0}\otimes A^h
\overset{\alpha_0}\otimes\bold k e$ and consider the composite
$D_{\alpha_0,\alpha_1}$ by the composite
$$
V\overset{\Delta_V}\to B\otimes V
\overset{\pi_{\alpha_0,\alpha_1}}\to
\bold k \overset{\alpha_0}\otimes A^h\overset{\alpha_0}\otimes Ve.
$$
By the associativity condition, the morphism
$D_{\alpha_0,\alpha_1}$ induces a morphism
$$
\underline{Hom}_{\Cal C_{\bold k}^{gr}}^1(V^{\alpha_0}e^{-\alpha_0},
A^h\otimes V^{\alpha_1}e^{-\alpha_1})=
\underline{Hom}_{\Cal C_{A}^{gr}}^1(V^{\alpha_0}e^{-\alpha_0},
V^{\alpha_1}e^{-\alpha_1})
$$
which is also denoted as $D_{\alpha_0,\alpha_1}$.
This defines a one to one correspondence 
$ob(B^h-comod)^b\to 
ob(K^b\Cal C^{gr}_{\bold k})$.
For the construction and the proof of the inverse correspondence 
and the homotopy equivalence is similar to those in
Section \ref{main section}, so we omit the
detailed proof.

(2) Using the following lemma, the proof of (2)
is similar.

\begin{lemma}
Let $A$ be a connected graded DGA. 
Then there exists a graded subalgebra $\overline{A}$
of $A$ such that (1) the natural inclusion $\overline{A}\to A$ is a
quasi-isomorphism, and (2)the conditions (p) and (r) 
in Lemma \ref{strictly connectedness} are satisfied.
\end{lemma}

\end{proof}

\section{Deligne complex}

\subsection{Definition of Deligne algebra}
Here we give an application of DG category to Hodge theory.
Let $X$ be a smooth irreducible variety over $\bold C$
and $\overline{X}$ be a smooth compactification of $X$ such that
$D=\overline{X}-X$ is a simple normal crossing divisor.
Let $\Cal U=\{U_i\}_{i\in I}$ be an affine covering of $\overline{X}$
indexed by a totally ordered set $I$, and
$\Cal U_{an}=\{U_{an,j}\}_{j\in J}$ 
be a topological simple covering of $X$ indexed by 
a totally ordered set $J$,
which is a refinement of $\Cal U\cap X$.
Assume that the map $J\to I$ defining the refinement preserves
the ordering of $I$ and $J$.

Let $\Omega_{\overline{X}}^{\bullet}(\log (D))$ be the sheaf of 
algebraic logarithmic de Rham complex of $\overline{X}$ along the 
boundary divisor $D$. Let $F^{\bullet}$ be the
Hodge filtration of $\Omega_{\overline{X}}^{\bullet}(\log (D))$:
$$
F^{i}:0\to \cdots \to 0\to 
\Omega_{\overline{X}}^{i}(\log (D))\to
\Omega_{\overline{X}}^{i+1}(\log (D))
\to \cdots.
$$
Then $F^{\bullet}$ is compatible with the product structure
of $\Omega_{\overline{X}}^{\bullet}(\log (D))$.
The Cech complex 
$\check{C}(\Cal U, \Omega_{\overline{X}}^{\bullet}(\log (D)))$
becomes an associative DGA by standard Alexander-Whitney 
associative product using the total order of $I$.
We define 
$$
A_{FdR,i}=\check{C}
(\Cal U, F^i\Omega_{\overline{X}}^{\bullet}(\log (D))).
$$
Then the product induces a morphism of complex
$$
A_{FdR,i}\otimes A_{FdR,j}\to A_{FdR,i+j}
$$
and by this multiplication $A_{FdR}=\oplus_{i \geq 0} A_{FdR,i}$
becomes a graded DGA, which is called
the algebraic de Rham DGA. We define the homogenized algebraic
de Rham DGA $A^h_{FdR}\in K\Cal C_{\bold k}^{gr}$ as 
$\oplus_{i \geq 0} A_{FdR,i}(-i)$.

Let $\Omega_{X_{an}}^{\bullet}$
be the analytic de Rham complex and  
$\check{C}(\Cal U_{an},\Omega_{X_{an}}^{\bullet})$ 
the topological Cech complex of
$\Omega_{X_{an}}^{\bullet}$.
Since the map $J\to I$ preserve the orderings,
we have a natural quasi-isomorphism of DGA's:
$$
\check{C}(\Cal U, \Omega_{\overline{X}}^{\bullet}(\log (D)))\to
\check{C}(\Cal U_{an},\Omega_{X_{an}}^{\bullet}).
$$
We set $A_{dRan,i}=
\check{C}(\Cal U_{an},\Omega_{X_{an}}^{\bullet})$
and
$A_{dRan}=\oplus_{i\geq 0}A_{dRan,i}$.
Then $A_{dRan}$ becomes a graded DGA
by Alexander-Whitney associative product,
which is called the analytic de Rham DGA. 
We also define
$A_{B,i}=\check{C}(\Cal U_{an},(2\pi i)^{i}\bold Q)$
and set
$A_{B}=\oplus_{i\geq 0}A_{B,i}$, which is called the
Betti DGA. It is a sub graded DGA of 
$A_{dRan}$. We define the homogenized analytic de Rham DGA
$A_{dRan}^h$ 
and the Betti DGA $A_B^h$ as 
$A^h_{dRan}=\oplus_i A_{dRan,i}(-i)$ and
$A^h_B=\oplus_i A_{B,i}(-i)$. 
\begin{definition}[Deligne algebra]
We define the Deligne algebra $A_{Del}=A_{Del}(X)$ of $X$
by the fiber product 
$A_{FdR}\times_{A_{dRan}}A_{B}$.
It is graded by $A_{Del,i}=A_{FdR,i}\times_{A_{dRan,i}} A_{B,i}$.
(See Definition \ref{def of fiber product of DGAs}
 for the definition of fiber
product of DGA's.)
\end{definition}
\begin{example}
In the case of $X=Spec(\bold C)$, the complexes $A_{Del}$ and
$A^h_{Del}$ are equal to
\begin{align*}
A_{Del}^0 & =\bold C\bigoplus \oplus_{i\geq 0}
\bold (2\pi i)^i\bold Q, \quad
A_{Del}^1  =\oplus_{i\geq 0}\bold C,  \\
A_{Del}^{h,0} & =\bold C\bigoplus \oplus_{i\geq 0} 
\bold (2\pi i)^i\bold Q(-i), \quad
A_{Del}^{h,1}  =\oplus_{i\geq 0}\bold C(-i).
\end{align*}
The differentials are the natural maps. 
In this case, the bar spectral sequence degenerates at $E_1$
and
$Gr(H^0(B(\epsilon_B\mid A_{Del} \mid \epsilon_B)))$
is isomorphic to the tensor algebra generated by
$\oplus_{i>0}\bold C/\bold (2\pi i)^i\bold Q$ over $\bold Q$.
\end{example}
\begin{remark}
\begin{enumerate}
\item
We can define $A_{dR}$ for a smooth variety over an arbitrary
subfield $K$ of $\bold C$.
\item
Actually, $A_{Del}(X)$ depends on the choice of the 
compactification $\overline{X}$,
and coverings $\Cal U$ and $\Cal U_{an}$. But the choice
of admissible proper morphisms and refinements 
does not affect up to quasi-isomorphism.
\item
We can replace the role $\bold Q$ in $A_{B}$ by an arbitrary
subfield $F$ of $\bold R$. The corresponding Deligne algebra
is denoted as $A_{Del,F}(X)$.
\end{enumerate}
\end{remark}

Let $p_{dR}$ (resp. $p_{B}$) be a $\bold C$-valued point of 
$X$ (resp. a
point in $X(\bold C)^{an}$). Then we have 
an augmentation $\epsilon_{dR}(p)$
(resp. $\epsilon_{B}(p)$) of $A_{dR}$ (resp. $A_{B}$).
The following proposition is a consequence of Proposition 
\ref{main theorem for graded case}.
\begin{proposition}
The category $H^0K^0C_{A_{Del}}^{gr}$
is equivalent to the category of homogeneous
$H^0(B^h(\epsilon_{B}\mid A_{Del}\mid \epsilon_B))$-comodules.
\end{proposition}
\subsection{Nilpotent variation of mixed Tate Hodge structure}
In this subsection, we prove that the category of nilpotent
variation of mixed Tate Hodge structures on
a smooth irreducible algebraic variety is equivalent to
that of comodules over $H^0(B^h(\epsilon \mid A_{Del}(X)\mid \epsilon))$.

\begin{definition}[VMTHS]
\label{VMTHS}
Let $X$ be a smooth irreducible algebraic variety over $\bold C$.
 A triple $(\Cal F, \Cal V,comp)$ of the following data
is called a variation of mixed Tate Hodge structure on $X$.
\begin{enumerate}
\item
A 4ple $\Cal F=(\Cal F,\nabla,F^{\bullet},W_{\bullet})$
is a locally free sheaf $\Cal F$ on $X$
with a connection logarithmic
connection
$$
\nabla:\Cal F \to \Omega_{\overline{X}}^1(\log(D))\otimes \Cal F
$$
and decreasing and increasing filtrations $F^{\bullet}$ 
and $W_{\bullet}$ on $\Cal F$ with the following properties.
(We assume that the filtration $W$ is indexed by even
integers.)
\begin{enumerate}
\item
$$
\nabla(F^i\Cal F)\subset 
\Omega_{\overline{X}}^1(\log(D))\otimes F^{i-1}\Cal F.
$$
\item
$$
\nabla(W_{-2j}\Cal F) \subset
\Omega_{\overline{X}}^1(\log(D))\otimes W_{-2i}\Cal F.
$$
\end{enumerate}
\item
A pair
$\Cal V=(\Cal V,W_{\bullet})$ is 
a local system with a filtration $W_{\bullet}$
on $X(\bold C)^{an}$.
\item
An isomorphism of local system
$$
\Cal V\otimes \bold C \overset{\simeq}\to (\Cal F\otimes 
\Cal O_{X(\bold C)^{an}})^{\nabla=0}
$$
on $X(\bold C)^{an}$ compatible with the filtrations $W_{\bullet}$
on $\Cal F$ and $\Cal V$.
\item
The fiber
of the associated graded object $(Gr_{-2i}^W\Cal F,
Gr_{-2i}^W\Cal V,comp)$ at $x$
is isomorphic to a sum of mixed Tate Hodge structure of
weight $-2i$ for all $x \in X(\bold C)^{an}$
\end{enumerate}
\end{definition}

\begin{definition}[NVMTHS]
\begin{enumerate}
\item 
A variation of mixed Tate Hodge structure is said to be constant if 
the local system $\Cal V$ is a trivial local system.
\item
A constant variation of mixed Tate Hodge structure is said to be 
split if it is isomorphic to a sum of pure Tate Hodge structures.
\item
A variation of mixed Tate Hodge structure is said to be nilpotent
(denoted as NVMTHS for short)
if there exists a filtration 
$(N_{\bullet}\Cal F,N_{\bullet}\Cal V,comp)$
of $(\Cal F, \Cal V,comp)$ such that
the associated graded object $(Gr_{p}^N\Cal F,
Gr_{p}^N\Cal V,comp)$ is 
a split constant variation of mixed Tate Hodge structures
on $X(\bold C)^{an}$.
\end{enumerate}
\end{definition}

\begin{theorem}
\label{theorem bar and NVMTHS}
The category $H^0K^0C_{A_{Del}}^{gr}$
is equivalent to $NVMTHS(X)$.
It is also equivalent to the category of homogeneous
$H^0(B^h(\epsilon_B\mid A_{Del}\mid \epsilon_B))$-comodules.
\end{theorem}
By applying the above theorem for the case $X=Spec(\bold C)$, 
we have the following corollary.
\begin{corollary}
The category of mixed Hodge structures is
equivalent to the category of homogeneous 
$H^0(B^h(\epsilon_{B}\mid A_{Del}
(\bold C)\mid \epsilon_B))$-comodules.
\end{corollary}

\subsection{Proof of Theorem \ref{theorem bar and NVMTHS}
}

We construct a nilpotent variation of mixed Tate Hodge
structure for an object 
$V=\{V^i\}$ in $K^0\Cal C_{A_{Del}}$.
We set $V^i=V^{ip}e^i(p)$,
$V=\oplus_iV^ie^{-i}$ and $V(p)=\oplus_iV^{ip}e^i(p)$.
Then we have $V=\oplus_{i,p}V^{ip}(p)=\oplus_pV(p)$.
Let $\Cal O,\Cal O_{an},\Omega^i$ and $\Omega^i_{an}$
denote $\Cal O_{\overline{X}},\Cal O_{X^{an}},
\Omega^i_{\overline{X}}(\log D)$ and $\Omega_{X^{an}}^i$.
For indices $s,t,u \in J$, the corresponding element 
in $I$ by the map $J\to I$ for refinement is also
denoted as $s,t,u$ for short.
The sum $D$ of the differential
$d_{ij}\in \underline{Hom}_{\Cal C_{A_{Del}}^{gr}}^1
(V^ie^{-i},V^je^{-j})$
defines an element in
\begin{align*}
\underline{Hom}_{\Cal C_{\bold k}^{gr}}^0(V, A^{h,1}_{Del}
\otimes V) =
&
\underline{Hom}_{\Cal C_{\bold k}^{gr}}^0(V, \oplus_i
\prod_{s}\Gamma(U_s,F^i\Omega^1)(-i)\otimes V) \\
\oplus&
\underline{Hom}_{\Cal C_{\bold k}^{gr}}^0(V, \oplus_i
\prod_{s,t}\Gamma(U_{s,t},F^i\Cal O)(-i)\otimes V) \\
\oplus&\underline{Hom}_{\Cal C_{\bold k}^{gr}}^0(V, 
\oplus_i \prod_{s,t}\Gamma(U_{an,s,t}, 
(2\pi i)^i\bold Q_B(-i))\otimes V) \\
\oplus&
\underline{Hom}_{\Cal C_{\bold k}^{gr}}^0(V, \oplus_i
\prod_{s}\Gamma(U_{an,s},\Cal O_{an})(-i)\otimes V). 
\end{align*}
As in the proof of Proposition \ref{DGcat to connection},
$\nabla=D+1\otimes d$ defines an integrable $A_{Del}$-connection
$V\to A^1_{Del}\otimes V$.
We write $D=(\{\omega_s\},\{f_{st}\},\{\rho_{st}\},
\{\varphi_s\})$ according to the above direct sum decomposition, i.e.
$$
\rho_{st} \in \underline{Hom}_{\Cal C_{\bold k}^{gr}}^0(V, \oplus_i \Gamma(U_{an,s,t}, 
(2\pi i)^i\bold Q_B(-i))\otimes V),
$$ 
etc. Then 
as an element of
\begin{align*}
 \underline{Hom}_{\Cal C_{\bold k}^{gr}}^0(V, A^{h,2}_{Del}
\otimes V) =
&
\underline{Hom}_{\Cal C_{\bold k}^{gr}}^0(V, \oplus_i
\prod_{s}\Gamma(U_s,F^i\Omega^2)(-i)\otimes V) \\
\oplus&
\underline{Hom}_{\Cal C_{\bold k}^{gr}}^0(V, \oplus_i
\prod_{s,t}\Gamma(U_{s,t},F^i\Omega^1)(-i)\otimes V) \\
\oplus&
\underline{Hom}_{\Cal C_{\bold k}^{gr}}^0(V, \oplus_i
\prod_{s,t,u}\Gamma(U_{s,t,u},F^i\Cal O)(-i)\otimes V) \\
\oplus&\underline{Hom}_{\Cal C_{\bold k}^{gr}}^0
(V, \oplus_i \prod_{s,t,u}\Gamma(U_{an,s,t,u}, 
(2\pi i)^i\bold Q_B(-i))\otimes V) \\
\oplus&
\underline{Hom}_{\Cal C_{\bold k}^{gr}}^0(V, \oplus_i
\prod_{s}\Gamma(U_{an,s},\Omega_{an}^1)(-i)\otimes V) \\
\oplus&
\underline{Hom}_{\Cal C_{\bold k}^{gr}}^0(V, \oplus_i
\prod_{s}\Gamma(U_{an,s,t},\Cal O_{an})(-i)\otimes V), 
\end{align*}
we have
\begin{align*}
dD=&
(\{d\omega_s\},
\{-\omega_t+\omega_s+df_{st}\},
\{-f_{tu}+f_{su}-f_{st}\},\\ &
\{-\rho_{tu}+\rho_{su}-\rho_{st}\},
\{\omega_s+d\varphi_s\},\{f_{st}-\rho_{st}+\varphi_t-\varphi_s\})
\end{align*}
and
\begin{align*}
D\circ D=&
(\{\omega_s \omega_s\},
\{-\omega_sf_{st}+f_{st}\omega_t\},
\{f_{st}f_{tu}\},\\ &
\{\rho_{st}\rho_{tu}\},
\{-\omega_s\varphi_s\},\{-f_{st}\varphi_t+\varphi_s\rho_{st}\}).
\end{align*}
Here the products are tensor products of composites of
endomorphisms of $V$ and exterior products of differential forms.
Then by the integrability condition, we have $dD=D\circ D$.

By this relation, we have cocycle relations
\begin{equation}
(1+f_{su})=(1+f_{st})(1+f_{tu}),\quad
(1+\rho_{su})=(1+\rho_{st})(1+\rho_{tu}),
\end{equation}
integrability of connections $d\omega_s=\omega_s\omega_s$,
and the following commutative diagrams:
\begin{equation}
\label{patching comparison}
\begin{matrix}
\oplus_pV(p)\otimes \Gamma(U_{an,s,t},(2\pi i)^{-p}\bold Q)
& \overset{1+\rho_{st}}\to & 
\oplus_pV(p)\otimes \Gamma(U_{an,s,t},(2\pi i)^{-p}\bold Q) \\
1+\varphi_t\downarrow & & \downarrow 1+\varphi_s\\
V\otimes \Gamma(U_{an,s,t},\Cal O_{an})& \overset{1+f_{st}}\to & 
V\otimes \Gamma(U_{an,s,t},\Cal O_{an}) \\
\end{matrix}
\end{equation}
\begin{equation}
\label{patch connection}
\begin{matrix}
V\otimes \Gamma(U_{s,t},\Cal O)& \overset{1+f_{st}}\to & 
V\otimes \Gamma(U_{s,t},\Cal O) \\
\omega_t+1\otimes d\downarrow & & \downarrow \omega_s+1\otimes d\\
V\otimes \Gamma(U_{s,t},\Omega^1)& \overset{(1+f_{st})\otimes 1}\to & 
V\otimes \Gamma(U_{s,t},\Omega^1) \\
\end{matrix}
\end{equation}
\begin{equation}
\label{horizontality}
\begin{matrix}
V\otimes \Gamma(U_{an,s},\Cal O_{an})& \overset{1\otimes d}\to & 
V\otimes \Gamma(U_{an,s},\Omega^1_{an}) \\
1+\varphi_s\downarrow & & \downarrow 1+\varphi_s\\
V\otimes \Gamma(U_{an,s},\Cal O_{an})& 
\overset{\omega_{s}+1\otimes d}\to & 
V\otimes \Gamma(U_{an,s},\Omega^1_{an}) \\
\end{matrix}
\end{equation}

By the cocycle relations, we have a local system
$\Cal V$ and locally free sheaf $\Cal F$ 
by patching constant sheaves $V$ and free sheaves
$V\otimes \Cal O$ on $U_{an,s}$ and $U_{s}$
by the patching data $1+\rho_{st}$ and $1+f_{st}$.
The connection $\omega_s+1\otimes d$ defines an integrable
connection on $V\otimes \Gamma(U_s,\Cal O)$ by the integrability
condition and it is patched together to a global connection
by the commutative
diagram (\ref{patch connection}). The local
sheaf homomorphisms $1+\varphi_s$ patched together into
a global homomorphism $comp$ of sheaves 
$\Cal V\to \Cal F\otimes \Cal O_{an}$
on $X_{an}$ by the commutative diagram (\ref{patching comparison}).
Via the sheaf homomorphism $comp$,
the local system $\Cal V\otimes \bold C$ 
is identified with the subshaef of horizontal sections 
of $\Cal F\otimes \Cal O_{an}$ by
the commutative diagram (\ref{horizontality}).

We introduce filtrations $F^{\bullet}$ and $W_{\bullet}$
on $U_s$ and $U_{an,s}$ by
\begin{align*}
&F^i \Cal F=\oplus_{p\geq i} V(-p)\otimes \Cal O, \\ 
&W_{2i} \Cal F=\oplus_{p\leq i} V(-p)\otimes \Cal O, \\
&W_{2i} \Cal V=\oplus_{p\leq i} V(-p). \\
\end{align*}
The filtration $F$ gives rise to a well defined filtration
on $\Cal F$ since the patching data $f_{ij}$ for $\Cal F$
is contained in $\Gamma(U_{s,t},A^{h,0}_{FdR})=
\Gamma(U_{s,t},\Cal O(0))$.

Since $\omega$ is contained in  
\begin{align*}
&
\bigoplus_pHom_{\bold k}(V(p),V(p)\otimes \Gamma(U_s,\Omega^1)) \\
& \oplus\bigoplus_p
Hom_{\bold k}(V(p),V(p+1)\otimes \Gamma(U_s,\Omega^1)(-1)),
\end{align*}
the Griffiths transversality condition in
Definition \ref{VMTHS} is satisfied.
It defines a mixed Tate Hodge structure at any point $x$
in $X(\bold C)^{an}$ by the definitions
of two filtrations. Therefore $(\Cal F, \Cal V,comp)$
defines a variation of mixed Tate Hodge structure on $X$.

We set $N^iV=\oplus_{i\leq p}V^p$. Then this filtration
defines a filtration on the variation of mixed Tate Hodge
structure $(\Cal F, \Cal V,comp)$ and the associated
graded variation is split constant mixed Tate Hodge structures
on $X$. Thus we have the required nilpotent variation of
mixed Tate Hodge structure $(\Cal F,\Cal V,comp)$ on $X$.

We can construct an inverse correspondence by
attaching an $A_{Del}$-connection to a nilpotent
variation of mixed Tate Hodge structures as follows.
We choose sufficiently
fine Zariski covering $\Cal U=\{U_i\}$ such that the restrictions
of $\Cal F$ are free on each $U_i$.
Using two filtrations $F$ and $W$, $\Cal F$ splits
into $\oplus_p\Cal F(p)$ as an $\Cal O$ module.
We choose a trivialization of $\Cal F(p)\mid_{U_i}$.
By taking a refinement $\Cal U_{an}$ of $\Cal U\cap X_{an}$,
we may assume that $\Cal U_{an}$ is a simple covering.
By restricting $U_{an,s}$, we choose a trivialization $1+\varphi_i$
of the local system $\Cal V$. At last we choose a nilpotent
filtration $N^{\bullet}$ compatible with the connection and
local system and a splitting of $N^{\bullet}$.
Using these data we construct
an integrable nilpotent $A_{Del}$-connection 
$D=(\{\omega_s\},\{f_{st}\},\{\rho_{st}\},
\{\varphi_s\})$ by the 
commutative diagram 
(\ref{patching comparison}),
(\ref{patch connection}),
(\ref{horizontality}).

By passing to homotopy equivalence classes of morphisms,
we have the required equivalence of two categories.

\end{document}